\documentclass[11pt]{amsart}
\usepackage{amsmath}
\usepackage{amssymb}
\usepackage{amsfonts}
\usepackage{epsfig}

\usepackage[all]{xy}
\newcommand{\CP}{\mathbb{CP}}

\newcommand{\C}{\mathbb{C}}
\newcommand{\CC}{\mathbb{C}}
\newcommand{\ZZ}{\mathbb{Z}}
\newcommand{\PP}{\mathbb{P}}
\newcommand{\FF}{\mathbb{F}}
\newcommand{\RR}{\mathbb{R}}
\newcommand{\Z}{\mathbb{Z}}

\newcommand{\A}{\mathcal{A}}

\newcommand{\LL}{\mathcal{L}}
\newcommand{\g}{\gamma}

\newtheorem{thm}{Theorem}[section]
\newtheorem*{thm*}{Theorem}
\newtheorem{corollary}[thm]{Corollary}
\newtheorem{lemma}[thm]{Lemma}
\newtheorem{properties}[thm]{Property}

\newtheorem*{lemma*}{Lemma}
\newtheorem{definition}[thm]{Definition}

\newtheorem{prs}[thm]{Proposition}
\newtheorem{remark}[thm]{Remark}
\newtheorem{assumption}[thm]{Assumption}
\newtheorem{notation}[thm]{Notation}

\newtheorem*{notation*}{Notation}

\begin{document}
\title [Structure of fundamental groups of CL arrangements]{On the structure of fundamental groups of conic--line arrangements having a cycle in their graph}

\author{Michael Friedman and David Garber}

\address{Michael Friedman, Institut Fourier, 100 rue des maths, BP 74, 38402 St Martin d'H\'eres cedex, France; Max Planck Institute for Mathematics, Vivatsgasse 7, 53111 Bonn, Germany}
\email{Michael.Friedman@ujf-grenoble.frþ}

\address{David Garber, Department of Applied Mathematics, Faculty of
  Sciences, Holon Institute of Technology, 52 Golomb st., PO
  Box 305, 58102 Holon, Israel}
\email{garber@hit.ac.il}

\begin{abstract}
The fundamental group of the complement of a plane curve is a very important topological invariant.
In particular, it is interesting to find out whether this group is determined by the combinatorics of the curve or not, and whether
it is a direct sum of free groups and a free abelian group, or it has a conjugation-free geometric presentation.

In this paper, we investigate the structure of this fundamental group when the graph of the conic-line arrangement is a unique cycle of length $n$ and the conic passes through all the multiple points of the cycle. We show that if
 $n$ is odd, then the affine fundamental group is abelian but not conjugation-free. For the even case, if $n>4$, then using quotients of the lower central series, we show that the fundamental group is not even a direct sum of a free abelian  group and free groups.

\end{abstract}

\maketitle

%\tableofcontents

\section{Introduction}

The fundamental group of the complement of a plane curve is a very important topological invariant.
For example, it is used to distinguish between curves that form a Zariski pair, which is a pair of curves having the same combinatorics but
non-homeomorphic complements in $\CC\PP^2$ (see \cite{AB-CR} for the exact definition and \cite{ABCT} for a survey). Moreover, the Zariski-Lefschetz hyperplane section theorem (see \cite{milnor})
states that
$\pi_1 (\CC\PP^N - S) \cong \pi_1 (H - (H \cap S)),$
where $S$ is a hypersurface  in $\CC\PP^N$ and $H$ is a generic 2-plane.
Since $H \cap S$ is a plane curve, the fundamental groups of complements of plane curves
can also be used for computing the fundamental groups of complements of hypersurfaces.
Note that when $S$ is a hyperplane arrangement, $H \cap S$ is a line arrangement in $\CC\PP^2$. Thus, one of the main tools for investigating the topology of hyperplane arrangements is the fundamental groups $ \bar{G} = \pi_1(\CC\PP^2- \LL)$ and $G = \pi_1(\CC^2 - \LL)$, where $\LL$ is an arrangement of lines.

One of the main questions arising in the research of hyperplane arrangements  is how does the combinatorics - in this case, the intersection lattice - of such an arrangement determine the fundamental group $G$ or the quotients, for example, of its lower central series $G_i/G_{i+1}$ (where $G_1 = G$ and $G_i = [G_{i-1},G]$). For example, when does the arrangement have a conjugation-free geometric presentation for its fundamental group? Also, it is well-known that for line arrangements, $G/G_2$,
$G/G_3$ and $G_2/G_3$ are determined by the combinatorics (see Section \ref{subsecCertainsQuotient}) and in fact Falk \cite{Falk} has shown that the rank of the quotients $G_i/G_{i+1}$ is also determined by the combinatorics. However, as Rybnikov shows \cite{R}, the quotient $G/G_4$ is not determined by the combinatorics, at least for complex arrangements.

These questions lead us to investigate the situation in the simplest generalization of arrangements of lines: conic--line (CL) arrangements. Indeed, some families of CL arrangements were studied by Amram et al. (see e.g. \cite{AT3,AT2} and especially \cite[Theorem 6]{AT1}). We also showed in  \cite{FG} that the combinatorics  of some families of real CL arrangements determines the structure of the corresponding fundamental group $G$ and that $G$ is conjugation-free. However, the family $\A_n$,  where the graph of the arrangement is a cycle of length $n$ and the conic passes through all the vertices of the graph, poses problems: not only that these arrangements are not conjugation-free (at least for odd $n$), but one has to differentiate between  cycles whose lengths have different parity.

In this paper, we give a complete description of the affine fundamental group $\pi_1(\CC^2 - \A_n)$ for the case of odd $n$: in this case, the fundamental group is abelian but not conjugation-free. For the case of
even $n$, we prove that the fundamental group is not abelian and not a direct sum of a free abelian  group and free groups. The last statement is proven by studying the groups $G_2/G_3$, $G/G_3$ and $Z(G/G_3)$.

\medskip

   The paper is organized as follows. In Section \ref{secLineArr}, we survey the known results on line arrangements, the conjugation-free property and certain quotients of the fundamental group arising from the lower central series. In Section \ref{secCycle34}, we examine two special cases, when the arrangement is as above and the  cycle is of  length $3$ or $4$, and in Section \ref{secGenGraph} we prove the main result:  while for odd $n$, the fundamental group is abelian and not conjugation-free, for even $n>4$, the fundamental group is not a sum of a free abelian group and free  groups.
 \medskip

{\textbf{Acknowledgements}}: We would like to thank Arkadius Kalka, Meital Eliyahu and Uzi Vishne for stimulating talks.

 The first author would like to thank the Max-Planck-Institute f\"ur Mathematik in Bonn for the warm hospitality and support and the Fourier Institut in Grenoble, where the final part of this paper was carried out.

\section{Arrangements and the conjugation-free property} \label{secLineArr}

In this section, we give a short survey of the known results concerning the structure of the fundamental group of the complement of  line arrangements and conic-line arrangements, while mentioning also  the conjugation-free property.

\subsection{Arrangements and their associated graphs}

An {\it affine line arrangement} in $\CC^2$ is a union of copies of $\C^1$ in $\C^2$. Such an arrangement is called {\em real} if the defining equations of all its lines can be written with real coefficients, and {\em complex} otherwise.

For a real or complex line arrangement $\mathcal L$, Fan \cite{Fa2} defined a graph $G(\mathcal{L})$ which is associated to the multiple points of $\LL$ (i.e. points where more than two lines are intersected). We give here its version for real arrangements (the general version is more delicate to explain and will be omitted): Given a real line arrangement $\mathcal L$, the graph $G(\mathcal{L})$ lies on the real part of $\mathcal L$. Its vertices are  the multiple points of $\mathcal{L}$ and its edges are  the
segments between the multiple points on lines which have at least two multiple points. Note that if the arrangement consists of three
multiple points on the same line, then $G(\mathcal{L})$ has three vertices on the same edge (see Figure \ref{graph_GL}(a)).
If two such lines happen to intersect in a \emph{simple} point (i.e. a point where exactly two lines are intersected), it is ignored
(i.e. there is no corresponding vertex in the graph). See another example in Figure \ref{graph_GL}(b) (note that Fan's definition gives a graph slightly different from the graph defined in \cite {JY,WY}).

\begin{figure}[!ht]
\epsfysize 4cm
\centerline{\epsfbox{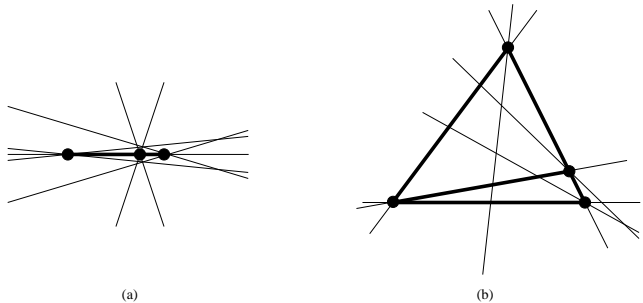}}
\caption{Examples for the graph $G(\mathcal L)$.}\label{graph_GL}
\end{figure}

\medskip

Fan \cite{Fa1,Fa2} proved that given  a complex line arrangement $\mathcal L$, if the graph $G(\mathcal L)$ has no cycles, then
$\pi_1 (\CC \PP ^2 - \mathcal L)$ is isomorphic to a direct sum of a free abelian group and free groups.
Eliyahu et al. \cite{ELST} proved the inverse direction to Fan's result (which was conjectured by Fan \cite{Fa2}), i.e. if the fundamental group of the arrangement is a direct sum of free groups and a free abelian group, then the graph $G(\mathcal L)$ has no cycles.

\medskip
We now turn to {\it real conic-line arrangements}.

\begin{definition} \label{defCLarr}
A {\em real conic-line (CL) arrangement} $\A$ is a union of conics and lines in $\CC^2$, where all the conics and the lines are defined over $\RR$ and every singular point (with respect to a generic projection) of the arrangement is in $\RR^2$. In addition, for every conic $C$, $C \cap \RR^2$  is not an empty set, neither a point nor a (double) line.
\end{definition}

Moreover, we assume from now on the following assumption:
\begin{assumption}\label{assume}
Let  $\A$ be a real CL arrangement. Then, for each pair of components $\ell_1,\ell_2$ of $\A$, $\ell_1$ and $\ell_2$ intersect transversally (i.e. the intersection multiplicity of $\ell_1,\ell_2$ is 1 at each intersection point).
\end{assumption}

For example, a tangency point is not permitted.

Similar to Fan's graph associated to line arrangements, one can associate the following graph to a real CL arrangement:

 \begin{definition}\label{defGraph}\emph{
\emph{The graph} $G(\A)$ for a real CL arrangement $\A$ is defined as follows: its vertices  will be the multiple points (with multiplicity larger than 2), and its edges will be the segments {\it on the lines} connecting these points if two such points are  on the same line
(see an example in Figure \ref{graphEx}).}
\end{definition}

\begin{figure}[!ht]
\epsfysize 4cm
\centerline{\epsfbox{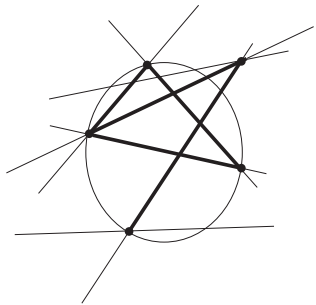}}
\caption{An example for the graph $G(\mathcal A)$ associated to  a CL arrangement
$\mathcal A$.}\label{graphEx}
\end{figure}

\medskip

In \cite{FG}, we have proved the equivalence  to Fan's result regarding real CL arrangements with one conic.

\begin{thm} \label{thmRealCLarrBetaZero}
Let $\A$ be a real CL arrangement with one conic and $k$ lines, such that $\beta (\A)=0$, where $\beta (\A)$ is the first Betti number of the associated graph $G(\mathcal A)$ (hence $\beta (\mathcal A)=0$
means that the graph $G(\mathcal A)$ has no cycles). Then $\pi_1 (\CC^2  - \mathcal A)$
 is isomorphic to a direct sum of a free abelian group and free groups. In addition, if the arrangement consists only of nodes and triple points and all the triple points are on the conic, then the corresponding fundamental
 group is abelian.
\end{thm}

Note that while for line arrangements the inverse direction (i.e. such a structure of the fundamental group implies that the associated graph has no cycles) is correct, for CL arrangements it is not true anymore. For example, take three generic lines and a circle passing through the three intersection points (see Figure \ref{3linesConic} below). The fundamental group of the complement of this arrangement is abelian (see \cite{Deg} and also Theorem \ref{prsG3G4}(a) below), although $\beta(\A)= 1>0$. We generalize this phenomenon here.

\subsection{Conjugation-free property}

Recall that for computing the fundamental group of a complement of a curve $C$ in $\C^2$, we use the Zariski-van Kampen thereom \cite{vK}. This theorem uses a generic projection $\pi : \C^2 \to \ell = \C^1$ (or a projection $\pi : \C\PP^2 \to  \C\PP^1$ with a center $O$) to a generic line $\ell$ (also called the \emph{reference line}) in order to induce the geometric generators in the fiber
$\C^1_p = \pi^{-1}(p)$, where $p$ is a generic point in $\ell$. These generators also generate $\pi_1 (\CC ^2 - C)$.

Using these notations,
we recall the notion of a {\it conjugation-free geometric presentation} for the fundamental group of line and CL arrangements (see \cite{EGT1,FG}):
\begin{definition}\label{CFGP-CLArr1}
Let $G$ be the fundamental group of the affine or projective complements of a real CL  arrangement with $k$ lines and $n$ conics (where $k>0$ and $n \geq 0$). We say that $G$ has {\em a conjugation-free geometric presentation} if $G$ has a presentation with the following properties:
\begin{itemize}
\item In the affine case, the generators $\{ x_1,\dots, x_{k+2n} \}$ are the meridians of lines and conics at $\C^1_p$, and therefore there are $k+2n$ generators.
\item In the projective case, the generators are the meridians of lines and conics at $\C^1_p = \pi^{-1}(p)$ except for one, and therefore there are  $k+2n - 1$ generators.
\item In both cases, the induced relations are of the following types:
$$x_{i_t} x_{i_{t-1}} \cdots x_{i_1} = x_{i_{t-1}} \cdots x_{i_1} x_{i_t} = \cdots = x_{i_1} x_{i_t} \cdots x_{i_2}$$
induced by an intersection point of multiplicity $t$, or
$$x_{i_1}=x_{i_2},$$
induced by a branch point,
where $\{ i_1,i_2, \dots , i_t \} \subseteq \{1, \dots, m \}$ is an increasing subsequence of indices,
where $m=k+2n$ in the affine case and $m=k+2n-1$ in the projective case. Note that if $t=2$ in the first type, we get the usual commutator.
\item In the projective case, we have an extra relation that a specific multiplication of all the generators is equal to the identity element.
\end{itemize}
Note that in each case we claim that with respect to  particular choices of the reference line $\ell$ (i.e. the line to which we project the arrangement), the point $p$ (the basepoint for both  the meridians in the fiber $\CC^1_p$ and the loops in the group $\pi_1(\ell-N,p)$)
and the projection point $O$, we have this conjugation-free property.
\end{definition}

\begin{remark}  \label{remProveLater}\emph{
%Later we prove that the choice of the base point $u$ is irrelevant (see Lemma \ref{lemAddLineComProof}) under some restrictions. Meanwhile,
In the model we work with, the reference line is $\ell = \{y = a\}, a\ll 0$, where $\ell$ is chosen to be below all the real singular points of the arrangement, the projection in $\CC^2$ is $(x,y) \rightarrow x$ (i.e. in $\CP^2$, the point $O$ is $(1:0:0)$) and the basepoint $p \in \ell$ is real.}
%(2) In \cite{FG} we proved that for ceratin real line arrangements and also for certain real CL arrangements, the conjugation-free property is independent of the basepoint $p$.}
\end{remark}

%\begin{notation} \label{relCyclic}
\begin{remark}\label{RemRelCyclic}
\emph{The relation of the form
$$x_{i_t}^{s_t} x_{i_{t-1}}^{s_{t-1}} \cdots x_{i_1}^{s_1} = x_{i_{t-1}}^{s_{t-1}} \cdots x_{i_1}^{s_1} x_{i_t}^{s_t} = \cdots = x_{i_1}^{s_1} x_{i_t}^{s_t} \cdots x_{i_2}^{s_2},$$
induced by an intersection point of multiplicity $t$, is called \emph{a cyclic relation of length} $t$ (where $s_i \in \langle x_1,\ldots,x_n \rangle$), and is abbreviated as
$$ [x_{i_1}^{s_1},\ldots,x_{i_t}^{s_t}] = e.$$
%\end{notation}
Note that a cyclic relation of length $t$ can be written as a list of $t-1$ commutative relations:
$$
[x_{i_k}^{s_k}, x_{i_{k-1}}^{s_{k-1}} \cdots x_{i_1}^{s_1} \cdot x_{i_t}^{s_t} \cdots x_{i_{k+1}}^{s_{k+1}}] = e,
$$
where $1 \leq k \leq t$.}
\end{remark}

We recall the following propositions from \cite{EGT2} and \cite{FG}:

\begin{prs} \label{prsCFless1}
(1) Let $\LL$ be a real line arrangement satisfying $\beta(\LL) \leq 1$. Then, $\pi_1(\CC^2-\LL)$ has a conjugation-free geometric presentation \cite{EGT2}.

(2) Let $\A$ be a real CL arrangement satisfying $\beta(\A) =0$. Then, $\pi_1(\CC^2-\A)$ has a conjugation-free geometric presentation \cite{FG}.
\end{prs}

The conjugation-free property is sometimes preserved while adding a line to the arrangement. Explicitly, we have the following proposition (see \cite{FG}):
\begin{prs} \label{lemAddLineComProof}
(1) Let $\LL$ be a real line arrangement such that  $\pi_1(\CC^2 - \LL,u)$ has a conjugation-free geometric presentation \emph{for any real basepoint} $u \in \ell - N$ (where $N$ is the set of the projection of singular points with respect to the projection $\pi$). Let $L$ be a line not in $\LL$ that passes through a single intersection point of $\LL$. Then,  $\pi_1(\CC^2 - (\LL \cup L),u)$ has a conjugation-free geometric presentation for any real basepoint $u$.

(2) Let $\A$ be a real CL arrangement with one conic such that  $\pi_1(\CC^2 - \A,u)$ has a conjugation-free geometric presentation \emph{for any  real basepoint} $u \in \ell - N$. Let $L$ be a line not in $\A$ that passes through a single intersection point of $\A$ such that $\beta(\A \cup L) = 0$. Then,  $\pi_1(\CC^2 - (\A \cup L),u)$ has a conjugation-free geometric presentation for any real basepoint $u$.
\end{prs}

Note  that Proposition \ref{lemAddLineComProof}(2) can be extended to some real CL arrangements with $\beta(\A \cup L) = 1$ (see \cite{FG}).

However, there are arrangements whose corresponding fundamental group does not have a conjugation-free geometric presentation, though it is ``almost" conjugation-free. Let us define this notion.
\begin{definition}\label{defAlmostCF}
Let $G = \pi_1(\CC^2 - \A)$ be a fundamental group of the affine complement of a real CL  arrangement $\A$ with $k$ lines and $n$ conics (where $k>0$ and $n \geq 0$). We say that $G$ has {\em an almost conjugation-free geometric presentation} if there is a geometric generator $x_i$, $1 \leq i \leq k+2n$, such that the image of $G$ under the epimorphism $x_i \mapsto e$ has a conjugation-free geometric presentation.
\end{definition}

The simplest example for this kind of arrangements is the \emph{Ceva arrangement} (also called the \emph{braid arrangement}) $\LL_3$; see Figure \ref{ceva3}. Using the package {\it TESTISOM} (see \cite{HR}), one can show that $\pi_1(\CC^2 - \LL_3)$ does not have a conjugation-free geometric presentation. Sending \emph{any} geometric generator to the identity element corresponds to deleting a line of $\LL_3$. Assuming we delete the line $L$ (see Figure \ref{ceva3}), the resulting arrangement $\LL_3' = \LL_3 - L$ will have $\beta(\LL_3')=0$ and thus, by Proposition \ref{prsCFless1}(1), $\pi_1(\CC^2 - \LL_3')$ has a conjugation-free geometric presentation.

\begin{figure}[!ht]
\epsfysize 3cm
\epsfbox{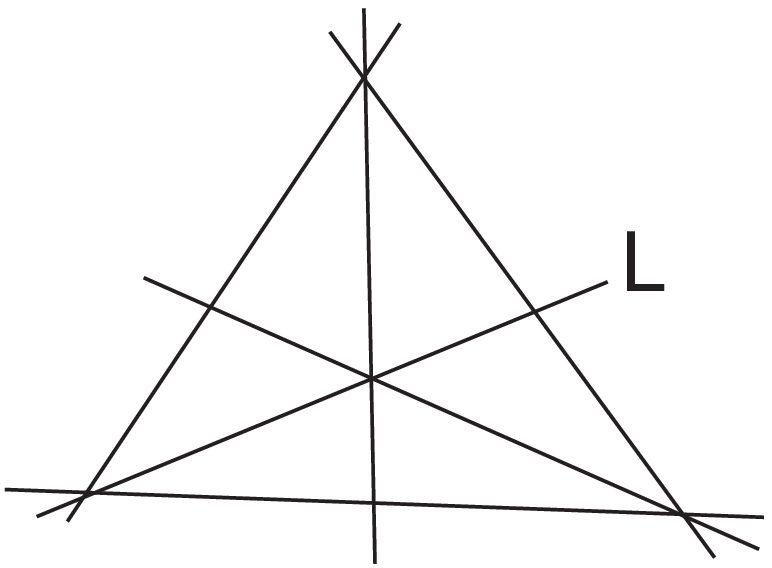}
\caption{$\LL_3$: the Ceva arrangement.}\label{ceva3}
\end{figure}

\subsection{On certain quotients of the fundamental group}\label{subsecCertainsQuotient}

Let $G = \pi_1(\CC^2 - \A)$, where $\A$ is either a line or a CL arrangement of degree $n$.
Denote by $G_2 = [G,G], G_3 = [G,[G,G]]$. In this section, we would like to review and study the structure of $G/G_2, G/G_3$ and $G_2/G_3$.

First, if $\A$ is a line arrangement then $G/G_2 \cong \Z^n$. If there are $k$ conics in the CL arrangement, then  $G/G_2 \cong \Z^{n-k}$.
Second, given  any group $H$, we have the following trivial properties for any $a,b,c \in H$:
\begin{properties}
\label{propH3}

\begin{enumerate}
\item $ [a,cbc^{-1}] \equiv [a,b]\, ({\rm{mod }}\, [H,[H,H]])$.
\item $[a,bc] \equiv [a,c][a,b]\, ({\rm{mod }}\, [H,[H,H]])$.
\item \{$[a,b,c] = e\} \equiv \{[b,a] = [c,b] = [c,a]^{-1}\}\, ({\rm{mod }}\, [H,[H,H]])$.
\end{enumerate}
\end{properties}

Thus, in $G/G_3$, according to Property \ref{propH3}(1) and Remark \ref{RemRelCyclic}, all the cyclic relations have no conjugations, i.e. $G/G_3$ has a conjugation-free geometric presentation, so it depends only on the combinatorics, i.e. the intersection lattice of the arrangement. See also
\cite{MS} for additional information regarding $G/G_3$ for line arrangements.

Turning to $G_2/G_3$, assume that $G$ is generated by $x_1,\ldots,x_n$. We know that $G_2/G_3$ is an abelian group, generated by the commutators $t_{i,j} \doteq [x_i,x_j]$ where $i<j$.  Note that $t_{i,j} = t^{-1}_{j,i}$. Thus $G_2/G_3$ is a quotient of $\Z^{ \binom{n}{2}}$. Moreover, if  the relation $[x_i,x_j^{\g}]=e$ holds in $G$, where $\g \in G$, then in $G_2/G_3$ we have that $t_{i,j} = e$ (by Property \ref{propH3}(1)).

Denoting $\phi_k \doteq $ rank$(G_k/G_{k+1})$, it is well-known that for a line arrangement $\A$, $\phi_k$ are determined by the combinatorics of the arrangement. Moreover, $\phi_2 = a_2$, where $a_i$ is the number of minimal generators of degree $i$ in the Orlik-Solomon ideal $I$, or that $a_2 = \binom{n}{2} - b_2$, where $b_2$ is the second Betti number of $\C^2 - \A$ (see  \cite{Falk}). We now give a different combinatorial description of $\phi_2$.

 Property \ref{propH3}(3) implies that every cyclic relation of length $3$ in $G$ (induced by  a triple point) is equivalent to an equality
of the form $t_{j,i} = t_{k,j} = t_{k,i}^{-1}$ in $G_2/G_3$. Thus, while every triple point contributes three generators to $G_2/G_3$, two of them can be expressed as the third (or as its inverse).

In the same way, while an intersection point $p$ of multiplicity $m$ contributes $\binom{m}{2}$ generators to $G_2/G_3$, $m-1$ of them can be expressed
as a product of the others; thus an intersection point of multiplicity $m$ contributes $v(p) \doteq \binom{m}{2} - m +1$ independent generators to $G_2/G_3$. Let $v(G) \doteq \sum_p v(p)$, where the sum goes over the intersection points of multiplicity $m \geq 2$.

 Note that for $\A$  a line arrangement, the generator $t_{i,j}$ appears as a term in the relations of $G_2/G_3$ only when describing the
 relations induced by the intersection of line $i$ and line $j$. Thus for every two intersection points, the generators $t_{i,j}$ contributed by them are different, and we see that $G_2/G_3 \cong \Z^{v(G)}$ (for a different point of view on the group $G_2/G_3$ in the case of line arrangements, see \cite{ELST}). However, this is already not true for CL arrangements, as can be seen in fourth step of Theorem \ref{prsGneven} below. That is, in CL arrangements, $G_2/G_3 \cong \Z^h$ where $h \leq v(G)$.

\begin{remark}
In the following sections, we use the braid monodromy techniques in order to compute the fundamental group of the complement of an arrangement. Since this material was covered extensively in numerous papers, we refer the reader to \cite{FG} for a survey of these methods.
\end{remark}

\section{A CL arrangement whose graph is a cycle of length $3$ or $4$}\label{secCycle34}
%\section{A graph with one cycle where the conic passes through all the vertices of the cycle} \label{secOneCyclePass}
In this section, we concentrate in two special cases of CL arrangements -- when the arrangement consists of one conic and three (or four) lines,
such that the conic passes through three (or four) intersection points, and the graph of the arrangement is either
a cycle of length $n=3$ or $n=4$. We are interested in whether the \emph{affine} fundamental group is either abelian or has a conjugation-free presentation, as these ``mini-examples" would serve as our guiding examples in the next section, where we investigate the general case.

\begin{notation}
Recall the notation
$[a,b,c]=e$ stands for the cyclic relations:
$cba = bac = acb$ (see Notation \ref{RemRelCyclic}).
\end{notation}

%\subsection{The cases $n=3$ and $n=4$}

It is known that for $n=3$, the projective fundamental group and  the affine fundamental group of this arrangement are abelian, by Degtyarev \cite{Deg}. Nevertheless, we prove it again  here, as the argument presented here will be generalized in Proposition \ref{prsGnOdd}.
We also present the case of $n=4$, which is the minimal example for even $n$.

Let $\A_3$ (resp. $\A_4$) be the following CL arrangement in $\C^2$: given three (four) generic real lines, draw a circle passing through the three (four) intersection points of the lines; for an illustration of $\A_3$, see Figure  \ref{3linesConic}, and for $\A_4$, see Figure \ref{4linesConic}.
Let $M_n \doteq \pi_1(\C^2 - \A_n)$ for $n=3,4$.

 \begin{prs} \label{prsG3G4}
 The groups $M_3, M_4$ are not conjugation-free. Moreover:
 \begin{enumerate}
 \item[(a)] $M_3$ is abelian.
 \item[(b)] $M_4 \cong \ZZ^3 \oplus \FF_2$.
 \end{enumerate}
 \end{prs}

 \begin{proof}
(a)  We look at Figure \ref{3linesConic}, where the dashed line is the initial fiber.

\begin{figure}%[!ht]
\epsfysize 5cm
\epsfbox{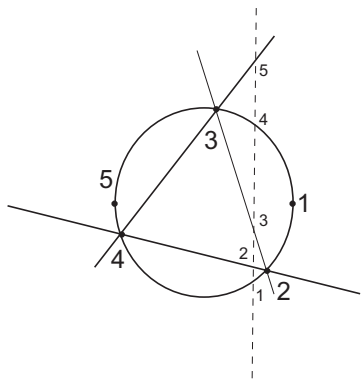}
\caption{A CL arrangement whose graph is a cycle of length 3: the small numbers stand for the numeration of the generators of $M_3$; the larger numbers stand for the numeration of the singular points.}\label{3linesConic}
\end{figure}

The following table describes how we compute the skeletons from the singular points of the arrangement with respect to the
projection: $j$ stands for the numeration of the images of the singular points  ($j \in \{1,\ldots,5\}$), $\lambda_{x_j}$ is the local Lefschetz skeleton describing which points, locally numerated, coincide when approaching the singular point $j$, and $\delta_{x_j}$ describes the Lefschetz diffeomorphism induced by a path going below the point $j$. Figure \ref{3linesConic_rel} presents the final skeletons of the singular points, induced by the Moishezon-Teicher method (see \cite{MoTe2} and a short survey in \cite[Section 2]{FG}).

\begin{center}
\begin{tabular}{|c|c|c|}
\hline
$j$ & $\lambda_{x_j}$ &  $\delta_{x_j}$
\\[0.5ex]\hline
1 & $[3,4]$ & $-$ \\
2 & $[1,3]$ & $\Delta \langle 1,3 \rangle$ \\
3 & $[3,5]$ & $\Delta \langle 3,5 \rangle$ \\
4 & $[1,3]$ & $\Delta \langle 1,3 \rangle$ \\
5 & $[3,4]$ & $-$ \\
\hline
\end{tabular}
\end{center}

\begin{figure}[!h]
\epsfysize 4cm
\epsfbox{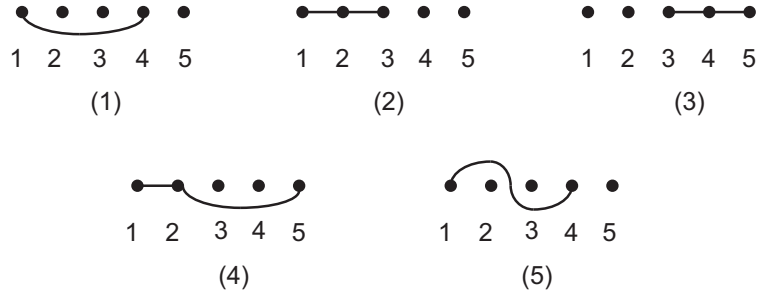}
\caption{The skeletons associated to the singular points (1)--(5).}\label{3linesConic_rel}
\end{figure}

By the Zariski-van Kampen theorem, we get that $M_3$ is generated by $5$ generators (denoted by $x_1,x_2,\dots,x_5$) with the following relations:
\begin{enumerate}
\item $x_1=x_4$,
\item $[x_1,x_2,x_3]=e$,
\item $[x_3,x_4,x_5]=e$,
\item $[x_1,x_2,x_5]=e$,
\item $x_2x_1x_2^{-1}=x_4$.
\end{enumerate}
%where the relation $[x,y,z]=e$ stands for the cyclic relation: $$zyx=yxz=xzy.$$

By the first and the last relations, we get that $[x_1,x_2]=e$ and thus the  relations in (4) are decomposed into three commutating relations:
$$[x_1,x_2] = [x_1,x_5] = [x_2,x_5] = e,$$ which in turn dissolve the other two cyclic relations into commutative relations. Thus $M_3$ is abelian and isomorphic to the abelian group $\ZZ^4$.

\medskip

The group $M_3$ is not conjugation-free, since if it were, then relation (5) should have been the  relation  $x_1=x_4$ in the above presentation. Therefore, $M_3$  would have been isomorphic, by Definition \ref{CFGP-CLArr1}, to the following group:
$$
M_3^{\text{cf}} = \langle x_1,x_2,x_3,x_5 : [x_1,x_2,x_3] = [x_3,x_1,x_5] = [x_1,x_2,x_5] = e \rangle.
$$
However, while $M_3$ is abelian, $M_3^{\text{cf}}$ is not$^1$\stepcounter{footnote}\footnotetext{We thank Uzi Vishne for giving this argument.}: Let $H \doteq M_3^{\text{cf}}.$ Then $H/([H,[H,H]])$ is a central but non-abelian extension of $H/[H,H] = \Z^4$ by $[H,H]/[H,[H,H]] = \Z$ (for computing the last equality we used Property \ref{propH3}(3)); thus $H$ is not abelian (see the fourth step of the proof of Theorem \ref{prsGneven} for an extended explanation for these kinds of arguments).

\medskip

\noindent
(b) Let us compute $M_4$. Let us look at Figure \ref{4linesConic}.

\begin{figure}[!h]
\epsfysize 7cm
\epsfbox{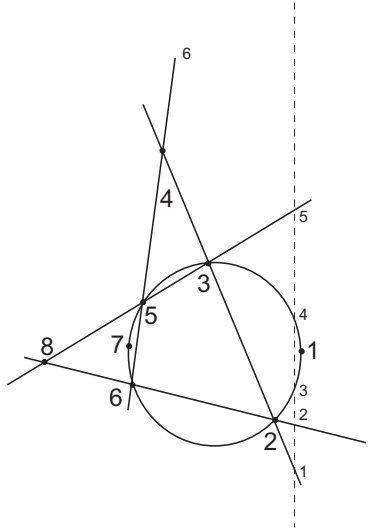}
\caption{A CL arrangement whose graph is a cycle of length 4: as before the small numbers stand for the numeration of the  generators of $M_4$; the larger numbers stand for the numeration of the singular points.}\label{4linesConic}
\end{figure}

Again, the following table describes how we compute the  relations induced by the singular points of the arrangement with respect to the
projection: $j$ stands for the numeration of the images of the singular points, $1 \leq j \leq 8$, $\lambda_{x_j}$ is the local Lefschetz skeleton, describing which points, locally numerated, coincide when approaching the singular point $j$, and $\delta_{x_j}$ describes the Lefschetz diffeomorphism induced by a path going below the point $j$. Figure \ref{4linesConic_final} presents the  skeletons of the singular points (3)--(8), induced by the Moishezon-Teicher method.

\begin{center}
\begin{tabular}{|c|c|c|}
\hline
$j$ & $\lambda_{x_j}$ &  $\delta_{x_j}$
\\[0.5ex]\hline
1 & $[3,4]$ & $-$ \\
2 & $[1,3]$ & $\Delta \langle 1,3 \rangle$ \\
3 & $[3,5]$ & $\Delta \langle 3,5 \rangle$ \\
4 & $[5,6]$ & $\Delta \langle 5,6 \rangle$ \\
5 & $[3,5]$ & $\Delta \langle 3,5 \rangle$ \\
6 & $[1,3]$ & $\Delta \langle 1,3 \rangle$ \\
7 & $[3,4]$ & $\Delta^{\frac{1}{2}}_{\text{IR}} \langle 2 \rangle$ \\
8 & $[2,3]$ & $-$ \\
\hline
\end{tabular}
\end{center}

\begin{figure}[h!]
\epsfysize 5cm
\epsfbox{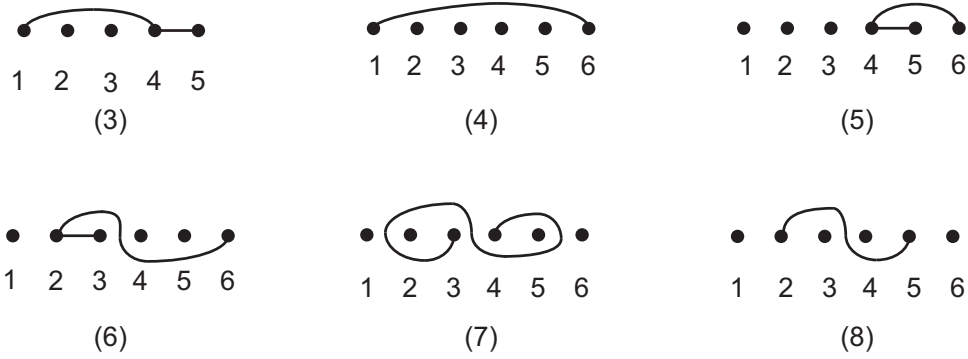}
\caption{The skeletons associated to the singular points  (3)--(8).} \label{4linesConic_final}
\end{figure}

Again, by the Zariski-van Kampen theorem, the group $M_4$ is generated by $6$ generators (denoted by $x_1,\dots,x_6$) with the following relations:
\begin{enumerate}
\item $x_3=x_4$,
\item $[x_1,x_2,x_3]=e$,
\item $[x_3 x_2 x_1 x_2^{-1} x_3^{-1},x_4,x_5]=e$,
\item $[x_6,x_5 x_4 x_3 x_2 x_1 x_2^{-1} x_3^{-1} x_4^{-1} x_5^{-1}]=e$,
\item $[x_5,x_5 x_4 x_5^{-1},x_6]=e$,
\item $[x_3,x_3 x_2 x_3^{-1},x_6]=e$,
\item $x_3 x_2 x_3 x_2^{-1} x_3^{-1} = x_5 x_4 x_5^{-1}$,
\item $[x_5,x_3 x_2 x_3^{-1}]=e$.
\end{enumerate}

\medskip

By relation (2), we get that relation (3) is equivalent to
$$[x_1,x_4,x_5] = [x_1,x_3,x_5]=e.$$
Using relations (2) and (3), we get that relation (4) is equivalent to $[x_1,x_6]=e$.

Moreover, relation (5) is equivalent to the following two relations:
$$x_6 x_5 x_4 x_5^{-1} x_5  =  x_5 x_4 x_5^{-1} x_5 x_6 ,\,\,\,\, x_5 x_4 x_5^{-1} x_5 x_6 = x_5 x_6 x_5 x_4 x_5^{-1}$$
or to
$$x_6 x_5 x_4 = x_5 x_4 x_6= x_4 x_6 x_5,$$
i.e., the relation is in fact equivalent to the cyclic relation $[x_4,x_5,x_6]=e$.
Similarly, relation (6) can be simplified to $[x_2,x_3,x_6]=e$.
Note that this is actually a general phenomenon, see Remark \ref{triple_rel_gen} below.

Multiplying relation (7) by $x_1$ from the left and using the fact that $x_4=x_3$, we get:
$$x_1 x_3 x_2 x_3 x_2^{-1} x_3^{-1} = x_1 x_5 x_3 x_5^{-1}.$$
Now use relations (2) and (3) to get that relation (7) is redundant.

\medskip

Hence, we get the following equivalent set of relations:
\begin{enumerate}
\item $x_3=x_4$,
\item $[x_1,x_2,x_3]=e$,
\item $[x_1,x_4,x_5]=e$,
\item $[x_1,x_6]=e$,
\item $[x_4,x_5 ,x_6]=e$,
\item $[x_2,x_3 ,x_6]=e$,
\item $[x_5,x_3 x_2 x_3^{-1}]=e$.
\end{enumerate}

\medskip

First, let us prove that $M_4$ is not conjugation-free. If it were, then relation (7), which is the only relation which has conjugations, would have been $[x_5,x_2]=e$. Denote by $M_4^{\rm cf}$ the group generated by 6 generators $x_1,\dots,x_6$ with the relations (1)--(6) and the relation $[x_5,x_2]=e.$ Using GAP \cite{GAP}, we find out that the number of epimorphisms of $M_4$ to the symmetric group $S_3$ is 3, whereas the number of epimorphisms of $M_4^{\rm cf}$ to $S_3$ is 1 (note that this fact already shows that $M_4$ is not abelian). Since the two groups are not isomorphic, it means that $M_4$ is not conjugation-free.

Note that although the above presentation of $M_4$  depends on the basepoint $u$ and on the reference line,  different choices
of them would induce an \emph{isomorphic} group, hence the number of epimorphisms to the symmetric group $S_3$ would remain the same (see \cite[Conjecture 2.16]{FG}).

\medskip

Second, we prove that $M_4 \cong \ZZ^3 \oplus \FF_2$.
Define: $x_{1'} = x_3 x_2 x_1$. Using $[x_1,x_2,x_3]=e$ and $x_1 = x_2^{-1} x_3^{-1} x_{1'}$, we get that $[x_{1'},x_2] = [x_{1'},x_3]=e$. Now, from $[x_1,x_3,x_5]=e$, we get
$[x_2^{-1} x_3^{-1}x_{1'},x_3,x_5]=e$, which induces the relation: $x_5 x_3 x_2^{-1} x_3^{-1} x_{1'} = x_3 x_2^{-1} x_3^{-1} x_{1'} x_5$. Using $[x_5,x_3 x_2 x_3^{-1}]=e$ (and thus
$[x_5,x_3 x_2^{-1} x_3^{-1}]=e$), we see that $[x_{1'},x_5]=e$. This means that
$$[x_2^{-1} x_3^{-1}x_{1'},x_3,x_5]=e\,\,\, \Rightarrow \,\,\,  [x_2^{-1} x_3^{-1},x_3,x_5]=e.$$ One of the relations induced by this relation is
$x_5x_3x_2^{-1}x_3^{-1} = x_2^{-1}x_3^{-1}x_5x_3$. As $[x_5,x_3 x_2^{-1} x_3^{-1}]=e$, we get that
$x_3x_2^{-1}x_3^{-1}x_5 = x_2^{-1}x_3^{-1}x_5x_3$ or $[x_3,x_2^{-1}x_3^{-1}x_5]=e$.

From the relation $[x_1,x_6]=e$, using $[x_2,x_3,x_6]=e$, we get that $[x_{1'},x_6]=e.$
Since $x_3=x_4$, we get that the new generator $x_{1'}$ commutes with all the other generators, i.e.
$M_4$ is isomorphic to the group $M_4^1$ generated by $5$ generators $x_{1'},x_2,x_3,x_5,x_6$ with the following relations:
\begin{enumerate}
\item $[x_{1'},x_i]=e,\, \mbox{ where } i \in \{ 2,3,5,6 \}$,
\item $[x_3,x_2^{-1}x_3^{-1}x_5]=e$,
\item $[x_3,x_5,x_6] = [x_2,x_3,x_6] = [x_5,x_3 x_2 x_3^{-1}] = e$.
\end{enumerate}

\medskip

Now, define $x_{2'} = x_3 x_2 x_3^{-1}$. Note that this substitution indicates that this group might not have a conjugation-free presentation (since we are using a different generator than the geometric meridian for the simplified presentation), as we have already shown. This means that $[x_5,x_{2'}]=e$ and $[x_3,x_{2'},x_6]=e$ (the last relation is in fact the second relation in relation (3) in the presentation of $M_4^1$). The relation $[x_3,x_2^{-1} x_3^{-1} x_5]=e$ is turned into
$[x_3, x_{2'}^{-1} x_5]=e$. Note that $[x_{1'},x_{2'}]=e$. Now, let $x_{2''} = x_{2'}^{-1} x_5$, so $x_{2'} = x_5 x_{2''}^{-1}$. Thus $[x_3,x_{2''}]=e$ and since $[x_5,x_{2'}]=e$, we get that $[x_5,x_{2''}]=e$. Note that $[x_{1'},x_{2''}]=e$. From the relation $[x_3,x_{2'},x_6]=e$, we get $x_6 x_{2'} x_3 = x_{2'} x_3 x_6$ or $x_6 x_5 x_{2''}^{-1} x_3 = x_5 x_{2''}^{-1} x_3 x_6$. Using $[x_3,x_{2''}]=e$, we get: $x_6 x_5 x_3 x_{2''}^{-1} = x_5 x_3 x_{2''}^{-1} x_6$. Since $[x_3,x_5,x_6]=e$ (i.e. $[x_6,x_5 x_3]=e$), we get that
$[x_6,x_{2''}]=e$. Therefore, $M_4^1$ is isomorphic to the group $M_4^2$ generated by $5$ generators $x_{1'},x_{2''},x_3,x_5,x_6$ where
$x_{1'}$ and $x_{2''}$ commute with all the other generators, and it has the additional cyclic relation $[x_3,x_5,x_6]=e$. Explicitly,
$$M_4 \cong \langle x_{1'} \rangle \oplus \langle x_{2''} \rangle \oplus \langle x_3,x_5,x_6 : [x_3,x_5,x_6]=e \rangle \cong
\ZZ \oplus \ZZ \oplus (\ZZ \oplus \mathbb{F}_2)
\cong \ZZ^3 \oplus \mathbb{F}_2.$$ \end{proof}

\begin{remark} \label{triple_rel_gen}
{\rm  As indicated in the proof above, note that whenever we have a cyclic relation induced by the set of paths appearing in Figures \ref{triple_rel_generic}(a) or \ref{triple_rel_generic}(b), then the cyclic relation is equivalent to the cyclic relation $[x_d,x_{d+1},x_c]=e$.}

\begin{figure}[h!]
\epsfysize 2cm
\epsfbox{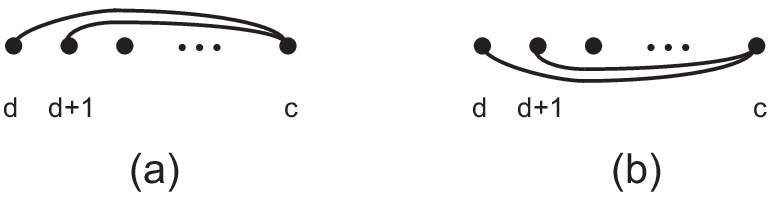}
\caption{}\label{triple_rel_generic}
\end{figure}

\end{remark}

% \begin{remark}
% {\em Note that $\pi_1(\CC\PP^2 - \mathcal{A}_4)$ is also not abelian: GAP \cite{GAP} shows that this group also has an epimorphism to $S_3$.}
% \end{remark}

The proof of Proposition \ref{prsG3G4} suggests that there is a distinction between the case of even $n$ and the case of odd $n$. As can be seen, the relation induced by the left branch point of the conic turns into a commutative relation in $M_3$, whereas it becomes trivial in $M_4$. This is the essential difference that distinguishes between $M_{2k}$ and $M_{2k+1}$.

\begin{remark}
\emph{Note that both groups $M_3$ and $M_4$ are \emph{almost-conjugation-free} (see Definition \ref{defAlmostCF}), as after sending any geometric generator, that corresponds to one of the lines, to the identity element, we get that the resulting arrangement has no cycles in its associated graph, and thus, by Proposition \ref{prsCFless1}(2), it is conjugation-free.
}\end{remark}

\section{A CL arrangement whose graph is a cycle: The general case}\label{secGenGraph}

In this section, we study the general case, whose specific cases were investigated in Section \ref{secCycle34}. We concentrate in the case of CL arrangements, having one conic, where the graph of this arrangement consists of one cycle of  length $n$, when the conic passes through all the point which correspond to the vertices of the graph. This means that the arrangement cannot be built inductively
using Proposition \ref{lemAddLineComProof}(2), i.e.
by adding, at each stage, a line that passes through only one singular point. Eventually, we would have to draw a line that passes through two intersection points, an operation which does not necessarily preserve the conjugation-free property.

We are interested in whether the \emph{affine} fundamental group is either abelian or has a conjugation-free presentation, as this would enable us to induce the structure of the fundamental group for more complicated arrangements. Indeed, as we saw in the previous section, for $n=3$, this group was abelian and for $n=4$, the group was a direct sum of a free abelian group and a free group, and both groups were not conjugation-free.

\medskip

 We take a regular $n$-gon in $\RR^2$ and we draw a circle through its $n$ vertices. We then extend the edges to be infinite straight lines, and we look at the resulting complexified arrangement in $\CC^2$. Thus we get a real CL arrangement $\mathcal{A}_n$ whose graph is a cycle of length $n$. Denote $M_n \doteq \pi_1(\CC^2 - \mathcal{A}_n)$.

%For proving that $M_n$ is abelian for  odd $n$,

In order to investigate $M_n$, we use a lemma which helps us to analyze the braid monodromy of the CL arrangement $\A_n$. We recall the Artin presentation of the braid group on $n+1$ strands:
$$B_{n+1} = \{ \sigma_1,\dots,\sigma_n : \langle \sigma_i, \sigma_{i+1} \rangle=e, \text{ for } 1 \leq i \leq n;\,\, [\sigma_i,\sigma_j]=e \text{ for } |i-j|>1 \},$$
where $\langle a,b \rangle \doteq abab^{-1}a^{-1}b^{-1}$.

\begin{lemma} \label{lemDelta}
Let $n \geq 2$, and
$$\Delta_{n+1} = \sigma_n(\sigma_{n-1} \sigma_n) (\sigma_{n-2} \sigma_{n-1} \sigma_n) \cdots (\sigma_1 \sigma_2 \sigma_3 \cdots \sigma_n)$$
be the Garside element in $B_{n+1}$ with respect to the Artin presentation. Define $\sigma_{n'} \doteq \sigma_{n+1} \sigma_n \sigma_{n+1} \in B_{n+2}$, and define
$$\Delta'_{n+1} = \sigma_{n'} (\sigma_{n-1} \sigma_{n'})(\sigma_{n-2}\sigma_{n-1} \sigma_{n'}) \cdots (\sigma_1 \sigma_2 \sigma_3 \cdots \sigma_{n'}) \in B_{n+2}.$$
Then:
$$\Delta'_{n+1} = \Delta_{n+2}\cdot \sigma_2 \sigma_3 \cdots \sigma_n.$$
\end{lemma}

\begin{proof}
The proof is by induction on $n$. We start with $n+1=3$:
$$\Delta'_3 = (\sigma_3 \sigma_2 \sigma_3) (\sigma_1 \cdot \sigma_3 \sigma_2 \sigma_3)
\overset{\sigma_3\sigma_2 \sigma_3 =\sigma_2\sigma_3 \sigma_2}{=}
(\sigma_3 \sigma_2 \sigma_3) (\sigma_1 \sigma_2 \sigma_3 \sigma_2) = (\sigma_3)(\sigma_2 \sigma_3) (\sigma_1 \sigma_2 \sigma_3) \sigma_2
=  \Delta_4 \cdot \sigma_2$$

Denote: $$\delta_{n+2} = \sigma_{n+1} \sigma_n \sigma_{n-1} \cdots  \sigma_2 \sigma_1 \in B_{n+2}.$$
Now, assume  that $\Delta'_n = \Delta_{n+1}\cdot \sigma_2 \sigma_3 \cdots \sigma_{n-1}$, and compute:
\begin{eqnarray*}
\Delta'_{n+1} &=& \sigma_{n'} (\sigma_{n-1} \sigma_{n'})(\sigma_{n-2} \sigma_{n-1} \sigma_{n'}) \cdots (\sigma_1 \sigma_2 \sigma_3 \cdots \sigma_{n'})=\\
&=& (\sigma_{n+1} \sigma_n \sigma_{n+1})(\sigma_{n-1} \cdot \sigma_{n+1} \sigma_n \sigma_{n+1})(\sigma_{n-2} \sigma_{n-1} \cdot \sigma_{n+1} \sigma_n \sigma_{n+1}) \cdot \\
 &  &   \cdots (\sigma_1 \sigma_2 \sigma_3 \cdots \sigma_{n-1} \cdot \sigma_{n+1} \sigma_n \sigma_{n+1}) =\\
&=& \delta_{n+2}\cdot \sigma_{n+1} \cdot \\
 & &  \cdot  \big[ (\sigma_{n+1} \sigma_n \sigma_{n+1})(\sigma_{n-1} \sigma_{n+1} \sigma_n \sigma_{n+1}) \cdots (\sigma_2 \sigma_3 \cdots \sigma_{n-1} \sigma_{n+1} \sigma_n \sigma_{n+1}) \big],
\end{eqnarray*}
where  the last equality uses the fact that  if $i>j+1$, then $[\sigma_i,\sigma_j]=e$.

Observe that the braids in the squared brackets (in the right hand side of the last equation) do not affect the first strand since $\sigma_1$ does not appear there, and thus the expression in the brackets is an element in the image of the homomorphism $\varphi : B_{n+1} \to B_{n+2}$, defined by $\varphi(\sigma_i)=\sigma_{i+1}$, for all $1 \leq i \leq n$. Note that the expression in these brackets is $\varphi(\Delta'_n)$. Using the induction hypothesis that
$$\Delta'_n = \Delta_{n+1}\cdot \sigma_2 \sigma_3 \cdots \sigma_{n-1},$$
we get:
\begin{eqnarray*}
\Delta'_{n+1}& =&  \delta_{n+2}\sigma_{n+1} \cdot \varphi(\Delta_{n+1}\cdot \sigma_2 \sigma_3   \cdots \sigma_{n-1}) = \\
  & = &  \delta_{n+2} \varphi(\sigma_n \cdot \Delta_{n+1}) \sigma_3 \sigma_4 \cdots \sigma_n = \\
  & = &  \delta_{n+2} \varphi(\Delta_{n+1}\cdot \sigma_1) \sigma_3 \sigma_4 \cdots  \sigma_n = \\
  & = & \delta_{n+2} \varphi(\Delta_{n+1}) \cdot \sigma_2 \sigma_3 \sigma_4  \cdots  \sigma_n = \\
  & = & \Delta_{n+2}\cdot \sigma_2 \sigma_3 \cdots \sigma_n.
\end{eqnarray*}

\end{proof}

We  start with the odd case:

\begin{thm} \label{prsGnOdd}
For odd $n=2k+1>1$, the group $M_n$ is abelian.
\end{thm}

\begin{proof}
Let $P_n$ be a regular $n$-gon where $n=2k+1$, bounded by the circle $C = \{x^2+y^2=1\}$ in $\RR^2$, such that there is one edge parallel to the $x$-axis. Extend all the edges of $P_n$ to infinite straight lines, and rotate this arrangement by an $\varepsilon$ degrees ($0<\varepsilon  \ll 1$) clockwise, where the center of the rotation is at the point $(0,0)$. Denote by $\mathcal{A}_n$ the resulting complexified CL arrangement.

Let $\ell = \{y=a\}$, $a \ll 0$, be such that all the singular points of $\mathcal{A}_n$ are above $\ell$ (where $\ell$ is the reference line) and consider the projection $\pi:\C^2 \to \ell,$ defined by $ (x,y)\mapsto x $. Let $\text{Sing}$ be the images of the singular points of $\mathcal{A}_n$ with respect to $\pi$.  Numerate the images of the triple points from right to left and choose $p \in \ell$ such that   $p$ is between the image of the first and the second triple point, see Figure \ref{11linesConic_fig} for the case $n=11$. Note that we used the same approach for $n=3$ in the proof of Proposition \ref{prsG3G4}(a).
The point $p$ will be the basepoint for $\pi_1(\ell - {\rm Sing},p)$.

\begin{figure}[h!]
\epsfysize 8cm
\epsfbox{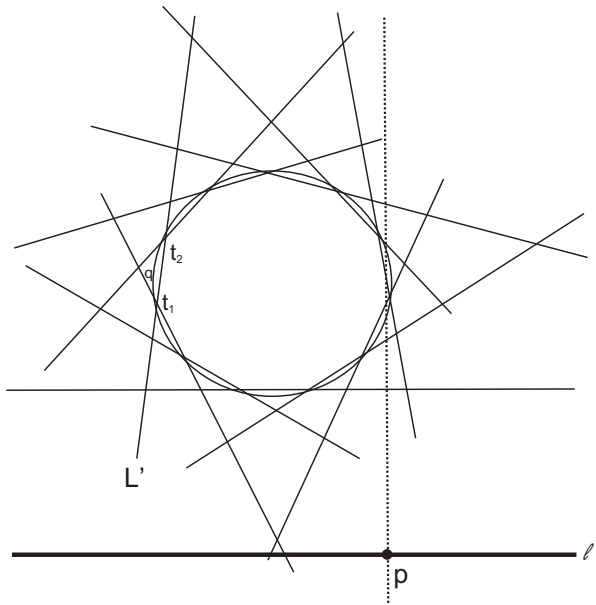}
\caption{The CL arrangement $\A_{11}$, where the point $q$ is the left branch point of the circle.}\label{11linesConic_fig}
\end{figure}

In order to compute $M_n$, we use the braid monodromy technique. We prove that $M_n$ is abelian in four steps: first, we  compute  the relation induced by the second branch point of the circle, denoted by $q$ (i.e. the branch point to the left of $p$). Second, we show that the triple points (except for the two leftmost triple points) always induce a relation without conjugations of the form $[a,b,c]=e$, where $a,b,c$ are geometric generators of $M_n$, induced by the base of $\pi_1(\pi^{-1}(p)~-~(\pi^{-1}(p) \cap \mathcal{A}_n))$.
Third, we show that the second branch point of the circle induces a commutative relation $[x,y]=e$, where $y$ is a generator corresponding to a line and $x$ is a generator corresponding to the circle. The fourth step shows that the combination of the former two steps dissolves one of the cyclic relations into commutative relations, and thus ``dissolving'' the cycle in the graph into a tree. Finally, this yields that $M_n$ is abelian.

As before note that, $x_1,\ldots,x_{2k+2}$ are the geometric generators of $M_n$.
\medskip

\noindent
{\bf First step:}
Let $q$ be the branch point of the circle (with respect to $\pi$) to the left of $p$. Then, the initial skeleton of $q$ is the segment $[k+2,k+3]$ and the first two braids that are applied on it are $\Delta \langle k,k+2 \rangle$ and $\Delta \langle k+2,k+4 \rangle$, corresponding to the two triple points $t_1,t_2$ that are located to the right of $q$ (with respect to the projection). Explicitly, $t_1$ is the first triple point to the right of $q$ (below the $x$-axis), and $t_2$ is the second (above the $x$-axis); see Figure \ref{11linesConic_fig} for the case  $n=11$. After this application, the skeleton looks  as in Figure \ref{initialSkel}(a).

 \begin{figure}[h!]
\epsfysize 3cm
\epsfbox{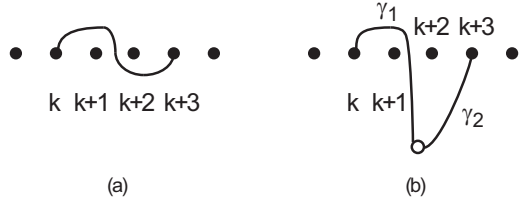}
\caption{The skeleton of the point $q$ after applying the first two braids.}\label{initialSkel}
\end{figure}

Now we have to check what is the effect of the Lefschetz diffeomorphism  (induced by the other nodes and triple points) on this skeleton (till the point $p$). We cut this skeleton into two paths, as depicted in Figure \ref{initialSkel}(b). We do that as all the braids that are induced by the singular points which are below the $x$-axis affect only the left path $\gamma_1$, while the braids  induced by the singular points which are above the $x$-axis  affect only the right path $\gamma_2$. Therefore, it  remains to find out what are the braids that are being applied on the left and  right paths.

We start by examining the effect of the braids induced by the singular points which are below the $x$-axis. Note that if we, for a moment, remove the circle $C$, then the sequence of braids we apply is in fact the sequence of braids we get by considering a generic   line arrangement of $k$ lines, composed of the $k$ lines passing through the triple points on the circle below the $x$-axis, i.e. the composition of the sequence of braids is in fact $\Delta_k \in B_k$ (where the group $B_k$ is generated by $\sigma_i,\, 1 \leq i \leq k-1$). Adding the section of the circle from $t_1$ till the right branch point (i.e. the branch point to the right of $p$) corresponds to replacing every instance of $\sigma_{k-1}$ by $\sigma_{k}\sigma_{k-1}\sigma_{k}$, since the circle only passes through the upper nodes of this generic arrangement.

Therefore, the sequence of braids being applied on $\gamma_1$ is $\Delta'_k$. By Lemma \ref{lemDelta}, this means that after applying $\Delta'_k = \Delta_{k+1}\cdot \sigma_2 \cdots \sigma_{k-1}$ on $\gamma_1$, we get the path presented in Figure \ref{FinalSkel_final}(a) (recall that we apply the braids from right to left).

\begin{figure}[h!]
\epsfysize 2.5cm
\epsfbox{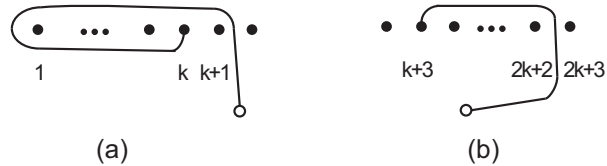}
\caption{The final parts of the skeleton associated to the point $q$.}
\label{FinalSkel_final}
\end{figure}

As for the braids that are being applied on the path $\gamma_2$, denote by $L'$ the line connecting  the two triple points
$t_1$ and $t_2$ (see Figure \ref{11linesConic_fig}). Note that the braids induced by the nodes that are on $L'$ and to the right of $t_1$ are being applied on $\gamma_2$ before any braid induced by any other singular point above the $x$-axis. But these braids do not affect $\gamma_2$: they all affect only the points with index greater than $k+3$. This means that we can ignore the line $L'$ when looking on the braids applied on $\gamma_2$, and hence  we can apply a similar argument to the one used for $\gamma_1$. We have again a generic arrangement of $k$ lines and a half-circle now passing through the lowest nodes. A similar check shows that after applying the sequence of the appropriate braids on $\gamma_2$, we get the path  depicted in Figure \ref{FinalSkel_final}(b). Thus, the final skeleton of the second branch point $q$ of the circle is presented in Figure \ref{FinalSkel}.

\begin{figure}[h!]
\epsfysize 1.5cm
\epsfbox{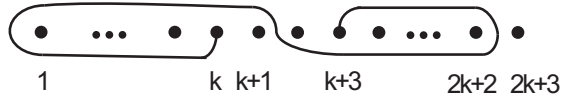}
\caption{The final skeleton associated to the branch point $q$.}
\label{FinalSkel}
\end{figure}

Therefore, the relation in $M_n$ induced by the second branch point $q$ is:
$$x_{k+1}x_k\cdots x_2x_1x_kx^{-1}_1x^{-1}_2 \cdots x^{-1}_kx^{-1}_{k+1} =x_{2k+2} \cdots x_{k+4}x_{k+3}x^{-1}_{k+4}\cdots x^{-1}_{2k+2},$$

\medskip

\noindent {\bf Second step:} We want to show that the relations induced by the triple points (except for the two leftmost two triple points) have no conjugations in $M_n$.
Again, we consider two cases: where the triple points (to the left of $p$) are below the $x$-axis, and where they are above.
 Obviously, the relations induced by the triple point to the right of $p$ have no conjugations, since this triple point is the closest point to the basepoint (from its right side), and hence there is no braid that is applied on its  initial skeleton.

Consider now the triple points (to the left of $p$) which are below the $x$-axis. Excluding the point $t_1$ from the computation (the reason for this will become clear in the third step), we numerate them from right to left by $p_1, \ldots ,p_{k-1}$. For the point $p_1$, its Lefschetz pair is $[k-1,k+1]$ and no braid is applied on it, so its induced relation is:
\begin{equation}\label{eqn_x}
[x_{k-1}, x_k, x_{k+1}]=e.
\end{equation}
Note that the  Lefschetz pair of each point $p_i$, $2 \leq i \leq k-1$, is also $[k-1,k+1]$ (see Figure \ref{11linesConic_fig2} for an illustration of  the case  $n=11$). For each $i$, let $\ell_{i_1},\ell_{i_2}$ be the two lines passing through $p_i$. For each $i$, $2 \leq i \leq k-1$, note that $\LL_i = \bigcup_{j = 1}^i \{\ell_{j_1},\ell_{j_2}\}$ is a generic arrangement of $i+1$ lines (as was already noted in the first step). Let us fix an $i$.

\begin{figure}[h!]
\epsfysize 8cm
\epsfbox{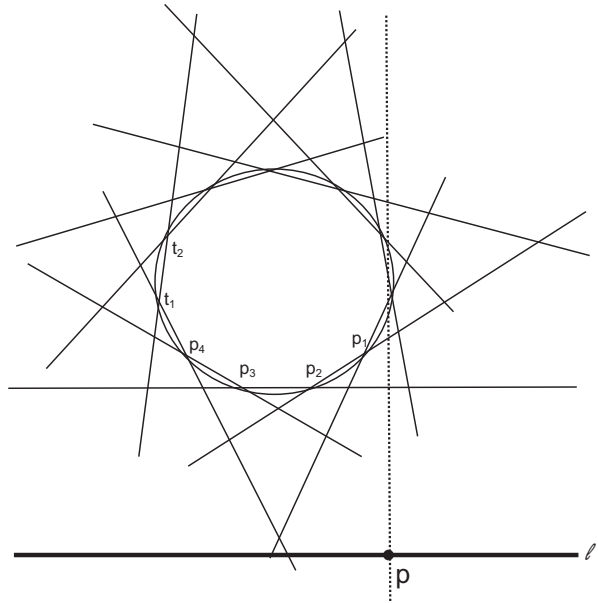}
\caption{The numeration of the triple points for the arrangement $\A_{11}$.}\label{11linesConic_fig2}
\end{figure}

If we consider only  the line arrangement $\LL_i$, then the Lefschetz pair of $p_i$ is $[i,i+1]$, and the composition of the
sequence of braids (in $B_{i+1}$) which are applied on this skeleton (in order to get the final skeleton) is in fact $\Delta_{i+1}\cdot \sigma^{-1}_i$ (we apply the braids on the skeleton from right to left). Note that in fact, applying $\sigma_i$ on the skeleton does not change it -- as this is a counterclockwise rotation of the points $i$ and $i+1$ by  $180^\circ$, so the skeleton remains the same. Thus, we can say that the braid that is applied on $[i,i+1]$ is $\Delta_{i+1}$. Now, drawing again the circle $C$ through $p_i$ till $p_{\, 1}$ corresponds to the fact that the Lefschetz pair (in the arrangement $\LL_i \cup C$) is now $[i,i+2]$ and, similar to the argument used in the first step, the braid that is applied on this skeleton  (in order to get the final skeleton) is $\Delta'_{i+1}$.
This means that the final skeleton (in the arrangement $\LL_i \cup C$) is  depicted in Figure \ref{FinalSkel3}(a).
Looking at the arrangement $\mathcal{A}_n$, let $n_i, n_i+1$ be the global numeration of the lines (in the fiber over $p$) passing through $p_i$. Thus, the final skeleton in $\mathcal{A}_n$ of the point $p_i$ is  depicted in Figure \ref{FinalSkel3}(b) (note that the
point $k$ in the fiber over  $p$ corresponds to the conic). Now, by Remark \ref{triple_rel_gen}, the  relation induced by $p_i$ is $[x_{n_i}, x_{n_i+1},x_k]=e$, as needed.

\begin{figure}[h!]
\epsfysize 2.5cm
\epsfbox{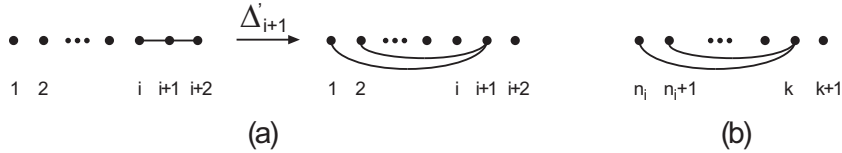}
\caption{The final skeleton of the point $p_i$.}
\label{FinalSkel3}
\end{figure}

The argument for the triple points above the $x$-axis (except for $t_2$) is almost the same, and hence omitted.

\medskip
\noindent {\bf Third step:}
Recall that the relation which was induced by the second branch point $q$ is:
\begin{equation} \label{eqnRelBranch}
x_{k+1}x_k\cdots x_2x_1x_kx^{-1}_1x^{-1}_2 \cdots  x^{-1}_kx^{-1}_{k+1}
= x_{2k+2} \cdots x_{k+4}x_{k+3}x^{-1}_{k+4}\cdots x^{-1}_{2k+2}. \end{equation}

Note also that the relation induced by the branch point of the circle to the right of $p$ is $x_{k} = x_{k+3}$.

\medskip

We now split our treatment according to the remainder of $n$ modulo 4. Since $n$ is odd, the remainder can be either 1 or 3.

\medskip

\noindent
\emph{Case (1)}:
Assume that $ n \equiv 3 (\text{mod } 4)$ (note that we can assume that $n>3$, i.e. $k>1$, since the case $n=3$ was already treated in Proposition \ref{prsG3G4}(a)). Then, the right hand side of Equation (\ref{eqnRelBranch}) is a conjugation of $x_{k+3}$ by the expression: $$x_{2k+2}x_{2k+1}\cdots x_{k+5} x_{k+4} = \prod_{m=2k+2,\ m\equiv 0({\rm mod} 2)}^{k+5} (x_m x_{m-1}).$$
Each pair of indices $m,m-1$
in the product corresponds to two consecutive lines (in the global numeration in the fiber over $p$), that intersect the circle in a triple point. By the second step, we know that  $[x_{m-1}, x_m,x_{k+3}]=e$ and therefore $x_m x_{m-1}$ commutes with $x_{k+3}$. This means that the right hand side can be simplified to $x_{k+3}$. Note that we have not used the cyclic relation induced by the triple point $t_2$, which involves (possibly  conjugations of) the generators $x_{k+3},x_{2k+2},x_{2k+3}$.

As for the left hand side of Equation (\ref{eqnRelBranch}), we can do the same procedure, until we get the following relation:
$$x_{k+1}x_kx_{k-1}x_{k-2}x_k x_{k-2}^{-1}x_{k-1}^{-1}x_k^{-1}x_{k+1}^{-1} = x_{k+3},$$ or
$$x_kx_{k-1}x_{k-2}x_k x_{k-2}^{-1}x_{k-1}^{-1}x_k^{-1} = x_{k+1}^{-1}x_{k+3}x_{k+1}.$$

Using the relation $[x_{k-2},x_{k-1},x_k]=e$ induced by the point $p_2$ (i.e. $x_{k-1}x_{k-2}x_k = x_kx_{k-1}x_{k-2}$), we get that:
$$x_k  = x_{k+1}^{-1}x_{k+3}x_{k+1} \text{ or } x_k  = x_{k+1}^{-1}x_{k}x_{k+1} \ \ \Rightarrow \ \ [x_k,x_{k+1}]=e.
$$

\medskip

\noindent
{\em Case (2)}:
Assume now that $ n \equiv 1 (\text{mod } 4)$. As in the previous case, the right hand side of Equation (\ref{eqnRelBranch}) can be simplified to
$x_{2k+2}x_{k+3}x_{2k+2}^{-1}$, and the left hand side can be simplified to $x_{k+1} x_k x_{k-1} x_k x_{k-1}^{-1} x_k^{-1} x_{k+1}^{-1}$, so we get:
$$
x_{k+1} x_k x_{k-1} x_k x_{k-1}^{-1} x_k^{-1} x_{k+1}^{-1} = x_{2k+2}x_{k+3}x_{2k+2}^{-1}.
$$

Now, the relation $[x_{k-1},x_k,x_{k+1}]=e$ is induced by the first triple point below the $x$-axis after the point $p$ (denoted by $p_1$ in the second step, see Equation (\ref{eqn_x})). Thus, we get:

$$x_{k+1} x_k x_{k-1} \cdot x_k x_{k-1}^{-1} x_k^{-1} x_{k+1}^{-1} = x_k \cdot x_{k-1} x_{k+1} x_k \cdot x_{k-1}^{-1} x_k^{-1} x_{k+1}^{-1} = $$
$$= x_k x_{k+1} x_k x_{k-1} x_{k-1}^{-1} x_k^{-1} x_{k+1}^{-1} = x_k.$$
Therefore:
$$x_k  = x_{2k+2}x_{k+3}x_{2k+2}^{-1} \ \ \Rightarrow  \ \ x_k  = x_{2k+2}x_{k}x_{2k+2}^{-1} \ \ \Rightarrow  \ \ [x_k, x_{2k+2}]=e.$$

Note that in both cases we got that the generator $x_k$ associated to the circle commutes with a generator associated to one of the lines.

\medskip

\noindent
{\bf Fourth step:} By the second step, $M_n$ has a cyclic relation of the form  $[x_k, x_{2k+1}, x_{2k+2}]=e$ and of the form $[x_{k-1},x_k,x_{k+1}]=e.$ In any case, the third step shows that one of these relations is dissolved into commutative relations: either to $[x_k, x_{2k+1}] = [x_k, x_{2k+2}] = [x_{2k+1}, x_{2k+2}]=e$ (in the case of $n\equiv 1$(mod $4$)) or to $[x_{k-1},x_k]= [x_{k-1},x_{k+1}]=[x_k,x_{k+1}]=e$ (in the case of $n\equiv 3$(mod $4$)). Denote by $t$ the triple point whose induced cyclic relation is dissolved and let $\ell_{t_1}, \ell_{t_2}$ be the two lines that pass through it. Let $t'$ be the other triple point that $\ell_{t_1}$ passes through it.
The fact that the cyclic relation turns into three commutative relations implies, from the perspective of the braid monodromy and the relations induced by it, that we can slightly rotate $\ell_{t_1}$ around $t'$ and still get an isomorphic fundamental group.

Let $U$ be a small neighborhood of $t$, i.e. $U \cap \mathcal{A}_n$ is an intersection of two lines and a circle at $t$. The slight rotation described above has the  effect on $U$ described in Figure \ref{nghbThreeLines}.

\begin{figure}[h!]
\epsfysize 2.5cm
\epsfbox{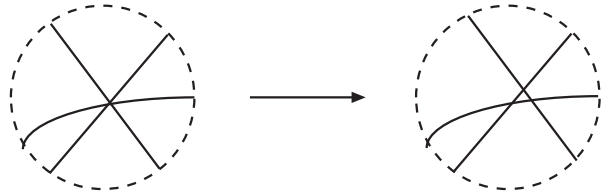}
\caption{The geometric effect of splitting the cyclic relation into commutative relations.}
\label{nghbThreeLines}
\end{figure}

However, this means that the graph of the revised arrangement has no cycles and all the triple points are on the conic, and by Theorem \ref{thmRealCLarrBetaZero}, the fundamental group is abelian, and we are done.
\end{proof}

\begin{corollary}
The arrangement $\A_{2k+1}$ is not conjugation-free, but is almost-conjugation-free.
\end{corollary}

\begin{proof}
Theorem \ref{prsGnOdd} implies that the arrangement $\A_{2k+1}$ is not conjugation-free: if it were, the second branch relation should have induced only the relation $x_{k} = x_{k+3}$ after the simplification process, and we would not be able to obtain that the generator associated to the circle commutes with a generator associated to one of the lines.

 Moreover, if we denote by $H$ the fundamental group $M_{2k+1}^{\text{cf}}$, which one would get from the arrangement $\A_{2k+1}$ if it were a conjugation-free arrangement, then we would get that $H$ is a central but non-abelian
extension of $H/[H,H] \cong \Z^{2k+2}$ by $[H,H]/[[H,H],H] \cong \Z$; thus, $H$ is not an abelian group, which is a contradiction (in order to see
more clearly that  $[H,H]/[[H,H],H] \cong \Z$, we refer the reader to the Fourth step of the proof of Theorem \ref{prsGneven}, where the same quotient is examined; thus we use the same methods from there).

 However, $\A_{2k+1}$ is almost-conjugation-free, as after sending any geometric generator, that corresponds to one of the lines, we get that the resulting arrangement has no cycle in its graph, and thus, by  Proposition \ref{prsCFless1}(2), it is conjugation-free. % Moreover, by the fourth step of Theorem \ref{prsGneven}, that would have meant that the fundamental group is not abelian.
\end{proof}

\begin{remark} \label{rem_CF_str}
{\rm
 Note that during the simplification process of the relation induced by the second branch point $q$, we did not use the relation induced by the unique triple point $t_0$ located to the right of the point $p$. This means that if we draw additional lines, passing only through $t_0$ (with a very negative slope), the simplification process of the resulting presentation will be identical, and therefore it implies that we can regard this new arrangement as an arrangement with $\beta(\A)=0$. Thus the fundamental group is a direct sum of a free group with $m(t_0)-2$ generators (induced by the singular point $t_0$, where  $m(t_0)$ is the multiplicity of $t_0$), and a free abelian group.}
\end{remark}

We now proceed to the even case.
As was proven in Proposition \ref{prsG3G4}(b), the group $M_4$ is not abelian. We generalize this fact in the following theorem:

\begin{thm} \label{prsGneven}
For even $n=2k>4$, the group $M_n$ is not abelian.
\end{thm}

\begin{proof}

We first change a bit the construction of the arrangement $\A_n$, such that there would be no parallel lines. Consider the circle $C = \{x^2+y^2=1\}$ and choose $k$ points, numerated clockwise by $v_1,\ldots,v_k$, on the halfcircle above the $x$-axis,  such that the distances $d(v_i,v_{i+1})$, $d(v_{i+1},v_{i+2})$ are the same for $1 \leq i \leq k-2$.
Draw a line through $v_1$  with a very positive slope $s_1$, that intersects the circle in an additional point (below the $x$-axis), denoted by $v'_1$.
Draw another line through $v_k$   with a very negative slope $s_2 \neq -s_1$, that intersects the circle in an additional point, denoted by $v'_k$.
Now, let $v'_2,\ldots,v'_{k-1}$ be another $k-2$ points on $C$, on the half-circle below the $x$-axis, such that the distances $d(v'_i,v'_{i+1})$, $d(v'_{i+1},v'_{i+2})$ are the same for $1 \leq i \leq k-2$.

We now have  $2k$ points on $C$. Connect them by segments and denote by
 $P_n$ the resulting $n$-gon which is bounded by the circle $C = \{x^2+y^2=1\}$ in $\RR^2$.
Extend all the edges of $P_n$ to infinite straight lines and denote by $\mathcal{A}_n$ the resulting complexified CL arrangement.

We use the same notations as in Theorem \ref{prsGnOdd}, i.e. the basepoint $p$ is chosen between the image of the first and the second
triple point (see Figure \ref{8linesConic_2} for the case $n=8$).
As before, $x_1,\ldots,x_{2k+2}$ are the geometric generators of $M_n$.

\begin{figure}[h!]
\epsfysize 8.5cm
\epsfbox{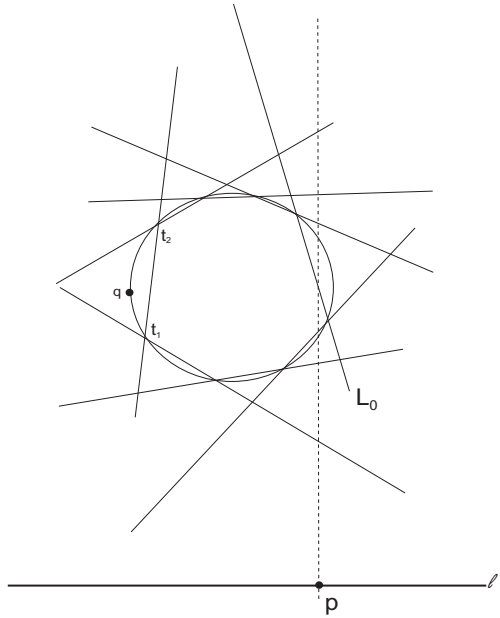}
\caption{The CL arrangement $\A_8$.}
\label{8linesConic_2}
\end{figure}

\medskip

In order to show that $M_n$ is not abelian, we use the braid monodromy technique as in the odd case. We prove this in four steps: first, we  compute  the relation induced by the second branch point of the circle, denoted by $q$ (i.e. the branch point to the left of $p$). Second, we show that the triple points (except for the two leftmost triple points) always induce a relation of the form $[a,b,c]=e$ (without conjugations), where $a,b,c$ are geometric generators of $M_n$, induced by the base of $\pi_1(\pi^{-1}(p)~-~(\pi^{-1}(p) \cap \mathcal{A}_n))$. Note that the first two steps
are parallel to the first two steps of Theorem \ref{prsGnOdd}. Third, we show that the second branch point of the circle $q$ induces  either a trivial relation (i.e. of the form $a = a$) or the same relation as induced by the first branch point, to the right of $p$. In the fourth step, we use the previous steps to show that $M_n$ is not abelian using its quotients.

\medskip
\noindent {\bf First step:}
Let $q$ be the branch point of the circle (with respect to the projection $\pi$) to the left of $p$. Then, the initial skeleton of $q$ is the segment $[k+1,k+2]$ and the first two braids that are applied on it are $\Delta \langle k-1,k+1 \rangle$ and $\Delta \langle k+1,k+3 \rangle$, corresponding to the two triple points $t_1,t_2$ that are located to the right of $q$ (with respect to the projection). Explicitly, $t_1$ is the first triple point to the right of $q$ (below the $x$-axis), and $t_2$ is the second (above the $x$-axis); see Figure \ref{8linesConic_2} for the case  $n=8$. After the application of these braids, the skeleton looks as in Figure \ref{initialSkel_even}(a).

 \begin{figure}[h!]
\epsfysize 3cm
\epsfbox{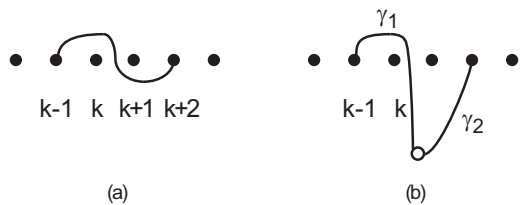}
\caption{The skeleton of the point $q$ after applying the first two  braids.}\label{initialSkel_even}
\end{figure}

Now, we have to check what is the effect of the Lefschetz diffeomorphism  (induced by the other nodes and triple points) on this skeleton (until the point $p$). However, we can proceed exactly as in the first step of the proof of Theorem \ref{prsGnOdd}. Explicitly, we cut this skeleton into two parts $\g_1$ and $\g_2$, as depicted in Figure \ref{initialSkel_even}(b).  By the same arguments as in
Theorem \ref{prsGnOdd}, the sequence of braids being applied on $\g_1$ is $\Delta'_{k-1}$, while
the sequence of braids being applied on $\g_2$ is $\Delta'_{k}$. This is since there are $k-2$ triple points below the $x$-axis, not including $t_1$ and the triple point to the right of $p$, while there are $k-1$ triple points above the $x$-axis, not including $t_2$.
This means that the final skeleton of the second branch point $q$ of the circle is presented in Figure \ref{FinalSkel_even}. Note that, in contrast to the case of odd $n$, the final skeleton of this case circumscribes  a different number of points in its left part than in its right part.

\begin{figure}[h!]
\epsfysize 1.5cm
\epsfbox{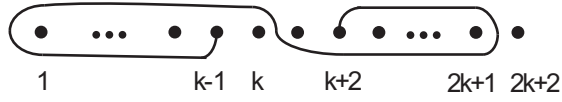}
\caption{The final skeleton associated to the branch point $q$.}
\label{FinalSkel_even}
\end{figure}

Therefore, the relation in $M_n$ induced by the second branch point $q$ is:
\begin{equation} \label{eqnRelBranch_even}
x_{k}x_{k-1}\cdots x_2x_1x_{k-1} x^{-1}_1x^{-1}_2 \cdots x^{-1}_{k-1} x^{-1}_{k} =x_{2k+1} \cdots x_{k+3}x_{k+2}x^{-1}_{k+3}\cdots x^{-1}_{2k+1}.
\end{equation}

\medskip

\noindent {\bf Second step:}
We want to show that the relations induced by the triple points (except for $t_1$ and $t_2$) have no conjugations; however, as can be easily seen, the proof of the corresponding step in
Theorem \ref{prsGnOdd} is independent of the parity of $n$ and thus we can use the same proof.

\medskip

\noindent {\bf Third step:}
Recall that the relation induced by the first branch point (to the right of $p$) is
\begin{equation}\label{relFirstBrPoint}
x_{k-1} = x_{k+2}.
\end{equation}
We now want to prove that relation (\ref{eqnRelBranch_even}) is either redundant or equivalent to
the relation induced by the first branch point. We again split our treatment according to the remainder of $n$ modulo~$4$. In this case, the remainder can be either $0$ or $2$.

\medskip

\noindent
\emph{Case (1)}: Assume that $n \equiv 2$(mod $4$). The right hand side of Equation (\ref{eqnRelBranch_even}) is a conjugation of $x_{k+2}$ by the expression: $$x_{2k+1}x_{2k}\cdots x_{k+4} x_{k+3} = \prod_{m=2k+1,\ m\equiv 0({\rm mod} 2)}^{k+4} (x_m x_{m-1}).$$
Each pair of indices $m,m-1$
in the product corresponds to two consecutive lines (in the global numeration in the fiber over $p$), that intersect the circle in a triple point. By the second step, we know that  $[x_{m-1}, x_m,x_{k+2}]=e$ and therefore $x_m x_{m-1}$ commutes with $x_{k+2}$. This means that the right hand side of Equation (\ref{eqnRelBranch_even}) can be simplified to $x_{k+2}$.

As for the left hand side of Equation (\ref{eqnRelBranch_even}), we can do the same procedure, until we get that the left hand side is equal to
$$x_{k} x_{k-1} x_{k-2}  x_{k-1} x_{k-2}^{-1} x_{k-1}^{-1} x_k^{-1}.$$

Now, by the cyclic relation $[x_{k}, x_{k-1}, x_{k-2}]=e$ induced by the first triple point to the left of $p$ below the $x$-axis, we have the following two relations: $$x_{k} x_{k-1} x_{k-2} = x_{k-1} x_{k-2} x_k \text{ and } x_{k} x_{k-1} x_{k-2}^{-1} = x_{k-2}^{-1} x_{k} x_{k-1}.$$
Therefore:

$$x_{k} x_{k-1} x_{k-2}  x_{k-1} x_{k-2}^{-1} x_{k-1}^{-1} x_k^{-1} =  x_{k-1} x_{k-2} x_k  x_{k-1} x_{k-2}^{-1} x_{k-1}^{-1} x_k^{-1} =$$$$
x_{k-1} x_{k-2} x_{k-2}^{-1} x_{k} x_{k-1} x_{k-1}^{-1} x_k^{-1} = x_{k-1},$$
and we get the known relation $x_{k-1} = x_{k+2}$.

\medskip
\noindent
\emph{Case (2)}: Assume that $n \equiv 0$(mod $4$). We now apply a different method than the method used in {Case (2)} of the third step of Theorem \ref{prsGnOdd}. Consider the nodes created by the intersection of the line $L_0$ numbered $k+1$ (in the fiber over $p$) with the other lines above the $x$-axis (see Figure \ref{8linesConic_2}). Using the Moishezon-Teicher algorithm, it is easy to show that the relations induced by these nodes
are
\begin{equation} \label{RelCommuteX}
[x_{k+1}, x_i]=e, \, k+4 \leq i \leq 2k+2.
\end{equation}

Now,  multiply relation (\ref{eqnRelBranch_even}) by $x_{k+1}$ from the left:
\begin{equation} \label{eqnRelBranch_even2}
x_{k+1}x_{k}x_{k-1}\cdots x_2x_1x_{k-1} x^{-1}_1x^{-1}_2 \cdots x^{-1}_{k-1} x^{-1}_{k} = x_{k+1}x_{2k+1} \cdots x_{k+3}x_{k+2}x^{-1}_{k+3}\cdots x^{-1}_{2k+1}.
\end{equation}

Let us look at the right hand side of relation (\ref{eqnRelBranch_even2}). By relations (\ref{RelCommuteX}), we can diffuse
$x_{k+1}$ till we get
\begin{equation} \label{eqnRelBranch_even2_rhs}
x_{2k+1}x_{2k} \cdots x_{k+5}x_{k+4} x_{k+1}x_{k+3}x_{k+2}x^{-1}_{k+3}x^{-1}_{k+4}\cdots x^{-1}_{2k+1}.
\end{equation}
Now, by the cyclic relation $[x_{k+1}, x_{k+2}, x_{k+3}]=e$ induced by the first triple point to the left of $p$ above the $x$-axis, we have that $x_{k+1}x_{k+3}x_{k+2} = x_{k+2}x_{k+1}x_{k+3} $, and thus  expression (\ref{eqnRelBranch_even2_rhs}) is transformed to
$$
x_{2k+1}x_{2k} \cdots x_{k+5}x_{k+4} x_{k+2}x_{k+1}x_{k+3}x^{-1}_{k+3}x^{-1}_{k+4}\cdots x^{-1}_{2k+1} $$$$ = x_{2k+1}x_{2k} \cdots x_{k+5}x_{k+4} x_{k+2}x_{k+1}x^{-1}_{k+4}\cdots x^{-1}_{2k+1} $$$$ \overset{\rm{Eqn.}\, (\ref{RelCommuteX})}{=} x_{2k+1}x_{2k} \cdots x_{k+5}x_{k+4}x_{k+2}x^{-1}_{k+4}\cdots x^{-1}_{2k+1}x_{k+1}.
$$

Now, by the cyclic relations, induced by the triple points above the $x$-axis, the generator $x_{k+2}$ commutes with $x_m x_{m-1}$ where
$m \in \{ 2k+1,2k-1,\ldots,k+5 \}$. Thus,
$$
x_{2k+1}x_{2k} \cdots x_{k+5}x_{k+4}x_{k+2}x^{-1}_{k+4}\cdots x^{-1}_{2k+1}x_{k+1} =$$$$ x_{k+2} x_{2k+1}x_{2k} \cdots x_{k+5}x_{k+4}x^{-1}_{k+4}\cdots x^{-1}_{2k+1}x_{k+1} = x_{k+2} x_{k+1}.
$$

That is, the right hand side of relation (\ref{eqnRelBranch_even2}) is equal to $x_{k+2} x_{k+1}$. Let us now deal with the left hand side
of relation (\ref{eqnRelBranch_even2}), which equals to:

$$
x_{k+1}x_{k}x_{k-1}x_{k-2}\cdots x_2x_1x_{k-1} x^{-1}_1x^{-1}_2 \cdots x^{-1}_{k-1} x^{-1}_{k}.
$$

Recall that the cyclic relation induced by the (unique) triple point to the right of $p$ is
$[x_{k-1},x_{k},x_{k+1}]=e$. Thus $$x_{k+1}x_{k}x_{k-1} = x_{k-1} x_{k+1}x_{k} \overset{\rm{Eqn.}\, (\ref{relFirstBrPoint})}{=} x_{k+2}  x_{k+1}x_{k}$$ and therefore relation (\ref{eqnRelBranch_even2}) gets the following form:

$$
x_{k+2}  x_{k+1}x_{k}x_{k-2}\cdots x_2x_1x_{k-1} x^{-1}_1x^{-1}_2 \cdots x^{-1}_{k-1} x^{-1}_{k} = x_{k+2} x_{k+1},
$$

\noindent
or to:
$$
x_{k-2}x_{k-3}\cdots x_2x_1x_{k-1} x^{-1}_1x^{-1}_2 \cdots x^{-1}_{k-2} x^{-1}_{k-1} = e.
$$

However, by  the cyclic relations, induced by the triple points below the $x$-axis, the generator $x_{k-1}$ commutes with $x_m x_{m-1}$ where $m \in \{ k-2,k-4,\ldots,2\}$, and therefore the above relation is transformed to
$$
x_{k-1} x_{k-2}x_{k-3}\cdots x_2x_1 x^{-1}_1x^{-1}_2 \cdots  x^{-1}_{k-2} x^{-1}_{k-1} = e \,\,\,\, \Rightarrow  \,\,\,\, e=e.
$$

\medskip
\noindent {\bf Fourth step:}
We now show that $H_n \doteq [M_n,M_n]/[M_n,[M_n,M_n]]$ is  isomorphic to $\ZZ$. This would imply that the commutator of $M_n$ contains non-trivial
elements of $M_n$, i.e. $M_n$ is not abelian.
%
%First, given  any group $H$, we have the following trivial properties for any $a,b,c \in H$:
%\begin{properties}
%\label{propH3}
%
%\begin{enumerate}
%\item $ [a,cbc^{-1}] \equiv [a,b]\, ({\rm{mod }}\, [H,[H,H]])$.
%\item $[a,bc] \equiv [a,c][a,b]\, ({\rm{mod }}\, [H,[H,H]])$.
%\item \{$[a,b,c] = e\} \equiv \{[b,a] = [c,a] = [c,b]^{-1}\}\, ({\rm{mod }}\, [H,[H,H]])$.
%\end{enumerate}
%\end{properties}

Note that besides the two relations induced by the branch points, all the other relations of $M_n$ are either
cyclic or commutative relations, possibly with conjugations (where the cyclic relations are induced only by the triple points). Property \ref{propH3}(1) implies that $M_n/[M_n,[M_n,M_n]]$ is conjugation-free (and thus also $H_n$) and
that every cyclic relation, induced by a triple point, can be presented as three equalities between commutators, by Property \ref{propH3}(3).

Recalling from Section \ref{subsecCertainsQuotient}, $H_n$ is generated
by the commutators $t_{i,j} \doteq [x_i,x_j]$ where $i<j$.  Note that $t_{i,j} = t^{-1}_{j,i}$. Moreover, if the  relation  $[x_i,x_j^{\g}]=e$ holds in $M_n$, where $\g \in M_n$, then in $H_n$ we have that $t_{i,j} = e$ (by Property \ref{propH3}(1)).

Moreover, since $\A_n - C$ is a generic  arrangement of $2k$ lines, all the relations in $\pi_1(\C^2 - (\A_n - C),p)$ are commutative relations, and therefore, the only relations in $M_n$ that do not vanish in $H_n$ come from the triple points, inducing cyclic relations. Recall also that in any case, $x_{k-1} = x_{k+2}$.
 Let us  list all the cyclic relations in $H_n$ (see also Figure \ref{10linesConic} for the case $2k=10$):

\begin{enumerate}
\item[(i)] Relations including the generator $x_{k+1}$:\\
(1) Below the $x$-axis: $$[x_{k-1},x_k,x_{k+1}] = e,$$
(2) Above the $x$-axis: $$[x_{k+1},x_{k+2},x_{k+3}] = e.$$
\item[(ii)] (Other) Relations induced by points above the $x$-axis: $$[x_{k+2},x_{k+3},x_{k+4}] = [x_{k+2},x_{k+4},x_{k+5}] = \cdots = [x_{k+2},x_{2k+1},x_{2k+2}] = e.$$
\item[(iii)] Another relation (induced by the point $t_1$) including the generator $x_{2k+2}$: $$[x_1,x_{k-1},x_{2k+2}] = e.$$
\item[(iv)]  Relations induced by points below the $x$-axis: $$[x_1,x_2,x_{k-1}] = \cdots = [x_{k-4},x_{k-3},x_{k-1}] = [x_{k-3},x_{k-2},x_{k-1}] = e.$$
\item[(v)] The  relation induced by the first triple point below the $x$-axis to the right of $p$: $$[x_{k-2},x_{k-1},x_k]=e.$$
\end{enumerate}

\begin{figure}[h!]
\epsfysize 5.5cm
\epsfbox{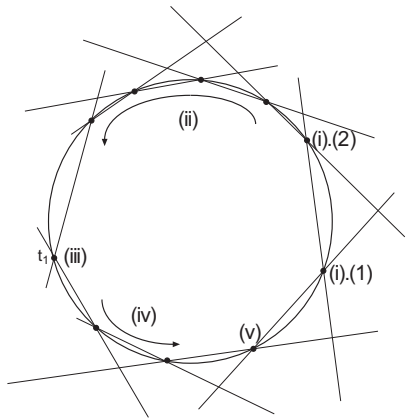}
\caption{An illustration for the  relations (with no conjugations in $H_n$) induced by the triple points, numerated counterclockwise, for the case $2k=10$.}
\label{10linesConic}
\end{figure}

Now, by Property \ref{propH3}(3) and by $x_{k-1} = x_{k+2}$, we see that in $H_n$ there is only one independent commutator, e.g. $y = t_{1,2}$, as all the others are either equal to it or to its inverse. Since the presentation of $M_n$ which was computed by the Zariski-van Kampen method is complete (i.e. there are no additional relations), we get that in $H_n$ there are no relations on $y$. Thus $H_n \cong \Z$ and therefore $M_n$ in not abelian.
\end{proof}

Note that in Proposition \ref{prsG3G4}(b), we proved that $M_4 \cong \Z^3 \oplus \FF_2$. We now prove that this is in fact an exceptional case.

\begin{thm} \label{prsGnevenDirectSum}
For even $n=2k>4$, the group $M_n$ is not isomorphic to a direct sum of a free abelian group and free groups.
\end{thm}
\begin{proof}

 We use the notations of the fourth step of Theorem \ref{prsGneven}.

We know that $M_{2k}/M_{2k}' \cong \ZZ^{2k+1}$.
Note that $H_{2k} \cong \Z$ for $k>2$ (by the fourth step of Theorem \ref{prsGneven}), means that $M_{2k}$ cannot be isomorphic to a direct sum of the form $ K = \ZZ^r \oplus \bigoplus_{i=1}^v \FF_{m_i}$ ($r\geq 1,m_i>1$) when $v>1$ or when $v=1$ and $m_1 > 2$, since in that case, $[K,K]/[K,[K,K]]$ would not be isomorphic to $\Z$. Note that the only case when $[K,K]/[K,[K,K]]$ might be isomorphic to $\Z$ is when $v=1$ and $m_1=2$ (by Witt's formula for the  rank of the $k$-th quotient in the lower central series of the free group, see \cite{MKS}), as we indeed get for $M_4$.

Assume by contradiction that  $M_{2k} \cong \ZZ^r \oplus \FF_2$. Thus $M_{2k}/M_{2k}' \cong \ZZ^{r+2} $ which implies that $ r = 2k-1$, and
$M_{2k} \cong \ZZ^{2k-1} \oplus \FF_2$. Denoting $K_{2k} \doteq M_{2k}/[M_{2k},[M_{2k},M_{2k}]]$, we get that $Z(K_{2k}) \cong \ZZ^{2k}$, since $Z(\FF_2/[\FF_2,[\FF_2,\FF_2]]) \cong \ZZ$ (where $Z(G)$ is the center of a group $G$). We will now contradict that deduction by showing that being in the center of $K_{2k}$ is equivalent to solving a homogeneous system of $2k+1$ linear equations with $2k+1$ variables.
We will get a contradiction by computing a lower bound on the rank of this system.

Let $X  = x_1^{m_1}\cdots  x_u^{m_u}\cdot \delta \in K_{2k}$ be a general word, where $\delta \in [M_{2k},M_{2k}]$, $u = 2k+2$ is the number of geometric generators of $M_{2k}$ and $m_i \in \ZZ$.
We want to see what are the conditions such that $X \in Z(K_{2k})$, i.e. what are the conditions on the exponents $m_i$ such that for any $1 \leq s \leq u, s \neq k+2$,

\begin{equation}\label{eqnXi}
(x_1^{m_1}\cdots  x_u^{m_u}\cdot \delta) \cdot x_s = x_s \cdot (x_1^{m_1}\cdots x_u^{m_u}\cdot \delta),
\end{equation}
\noindent
when we omit the case of $s=k+2$,  since $x_{k-1} = x_{k+2}$. Therefore, from now on, we consider the generators $x_{k-1}$ and $x_{k+2}$ as being the same, and thus, consider the expression $x_1^{m_1}\cdots x_u^{m_u}$ as equal to a product of $2k+1$ terms:
$x_1^{m_1}\cdots x_{k+1}^{m_{k+1}} \cdot \widehat{x_{k+2}^{m_{k+2}}} \cdot x_{k+3}^{m_{k+3}} \cdots x_u^{m_u}$
(when the exponents $m_{k-1}$ and $m_{k+2}$ are always added and thus considered as the exponent of $x_{k-1}$).

Note that $H_{2k}  \doteq [M_{2k},M_{2k}]/[M_{2k},[M_{2k},M_{2k}]]  \subseteq Z(K_{2k})$. Note also that the system of equations (\ref{eqnXi}) is a system of equations in $K_{2k}$. Therefore, $\delta \in Z(K_{2k})$ and thus can be omitted from equations (\ref{eqnXi}). Now, since $x_ix_j = t_{i,j}x_jx_i$, and using Property \ref{propH3}(2), the above system of equations is equivalent to:
\begin{equation} \label{eqnMatrixOverZ}
t_{s,1}^{m_1}\cdots \widehat{t_{s,k+2}^{m_{k+2}}} \cdots t_{s,u}^{m_u} = e,
\end{equation}
where $1 \leq s \leq u, s \neq k+2$. However, as already noted in the fourth step of Theorem \ref{prsGneven}, for every $i,j,k,l$,
$t_{i,j} = t_{k,l}^{\pm 1}$ in $H_{2k}$. Denoting $y = t_{1,2}$, we get that the system of equations (\ref{eqnMatrixOverZ}) is in fact a system of $2k+1$ linear equations in the variables $m_i$'s:
\begin{equation} \label{eqnMatrixOverZ2}
\alpha_{s,1}{m_1}+ \cdots + \alpha_{s,k+1}{m_{k+1}} + \alpha_{s,k+3}{m_{k+3}}  + \cdots + \alpha_{s,u}{m_u} = 0,
\end{equation}
where $\alpha_{i,j} \in \ZZ$. Note that since the $t_{i,j}$'s are equal to $y^{\pm 1}$, we have that $|\alpha_{i,j}| \leq 2$.

Now, the trivial solution to this homogeneous system of equations, i.e. $m_i = 0$ for all $i$, corresponds to the word $X = \delta  \in [M_{2k},M_{2k}]$. As
$H_{2k}= \langle y \rangle \cong \Z$, as was proved in the fourth step of Theorem \ref{prsGneven}, we get that
$\langle y \rangle \subseteq Z(K_{2k})$. Therefore, in order to prove that $Z(K_{2k}) \cong \ZZ^{2k}$, we have to show that there are
$2k-1$ other non-trivial linearly independent solutions for the system of equations (\ref{eqnMatrixOverZ2}).

The system of equations (\ref{eqnMatrixOverZ2}) has $2k+1$ linear equations with $2k+1$ variables. Once we show that there are three
independent equations at this system, this would mean that rank($A$)$\geq 3$ (where $A=$($\alpha_{i,j}$) is the coefficient matrix),  thus there are less than $2k-1$ non-trivial independent solutions, and we get a contradiction.

Let us examine more closely the system of equations (\ref{eqnMatrixOverZ}) for certain indices $s$.
For $s \neq k-1$, $t_{s,i} \neq e$ only if $i = k-1$ (since there are no nodes on the conic) or  the lines numerated $s$ and $i$ intersect at a triple point (on the conic).
Thus, for $s=1$, the corresponding equation, from the system of equations (\ref{eqnMatrixOverZ2}), is:
$$
\alpha_{1,2} m_2 + \alpha_{1,2k+2} m_{2k+2} + \alpha_{1,k-1} m_{k-1} = 0
$$
(where $\alpha_{1,2}, \alpha_{1,2k+2} \neq 0$). For $s=2k+2$, the corresponding equation is:
$$
\alpha_{2k+2,1} m_1 + \alpha_{2k+2,2k+1} m_{2k+1} + \alpha_{2k+2,k-1} m_{k-1} = 0
$$
(where  $\alpha_{2k+2,1}, \alpha_{2k+2,2k+1} \neq 0$). For $s = k-1$, $t_{s,i} = e$ only for $i=k-1$. Thus, for $s = k-1$  the corresponding equation is:
$$
\sum_{\substack{i=1,\\i\neq k-1,k+2}}^{2k+2} \alpha_{k-1,i} m_i = 0
$$
(where the coefficients in the above equation are not $0$). Obviously, the last three equations are linearly independent, and this finishes the proof.
\end{proof}

\begin{remark}
\rm{
%(1) Note that $H_{2k} \simeq \Z, k>2$ means that $M_n$ cannot be isomorphic to a direct sum of the form $ K = \ZZ^r \oplus \bigoplus_{i=1}^p \FF_{m_i}$ ($r\geq 1,m_i>1$) when $p>1$ or when $p=1$ and $m_1 > 2$, since in that case, $[K,K]/[K,[K,K]]$ would not be isomorphic to $\Z$. Note that the only case when $[K,K]/[K,[K,K]]$ might be isomorphic to $\Z$ is when $p=1$ and $m_1=2$, as we indeed get for $M_4$.
%
%(2) As confirmed by GAP \cite{GAP}, $M_6$ does not have any epimorphism to $S_3$, which implies that it is not isomorphic to a direct sum of abelian free group and free groups.
 $\A_{2k}$ is almost-conjugation-free, as after sending any geometric generator, that corresponds to one of the lines, to the identity element, we get that the resulting arrangement has no cycle in its graph, and thus, by Proposition \ref{prsCFless1}(2), it is conjugation-free. However, we do not know whether it is conjugation-free or not, but we conjecture that it is not conjugation-free based on the case of $\A_4$.
}
\end{remark}

\begin{remark}
\rm{
Using the result of Oka-Sakamoto \cite{OkSa}, one can generalize the above results to any CL arrangement whose graph is a disjoint union of
cycles.
}
\end{remark}

\end{document}